\documentclass{article}
\usepackage{amsfonts}
\usepackage{amssymb}
\usepackage{graphicx}
\usepackage{xcolor}
\usepackage{amsmath}
\usepackage[colorlinks=true,linkcolor=blue,citecolor=red,urlcolor=black]{hyperref}

\newtheorem{theorem}{Theorem}

\newtheorem{remark}[theorem]{Remark}

\newenvironment{proof}[1][Proof]{\noindent\textbf{#1.} }{\ \rule{0.5em}{0.5em}}

\title{Couette Taylor instabilities for counter-rotating cylinders in the small-gap regime}

\author{D. Bian\footnote{School of Mathematics and Statistics, Beijing Institute of Technology, Beijing 100081, China. Email: biandongfen@bit.edu.cn}, 
E. Grenier\footnote{Academy of Mathematics and System Science, Chinese Academy of Science, Beijing, China. Email:
emmanuelgrenier@amss.ac.cn},
G. Iooss\footnote{Laboratoire J.A.Dieudonné, I.U.F., Université Côte d’Azur, Parc Valrose, 06108 Nice Cedex 02, France. Email: iooss.gerard@orange.fr},
Z. Yang\footnote{Department of Mathematics, The Ohio State University, Columbus, OH 43210, USA. Email: yang.8242@osu.edu}
}
\date{\today}

\begin{document}

\maketitle


\subsubsection*{Abstract}

We study the Couette–Taylor instabilities for an incompressible viscous fluid between two coaxial cylinders of nearly equal radii, allowing counter-rotation with the ratio of rotation rate $\mu \in [-1,1]$. Working in a rotating frame and in a small-gap and small-viscosity regime, we derive the corresponding limiting Navier–Stokes system and analyze the linear stability of the Couette flow. In particular, we numerically compute the critical Taylor number for general perturbations and identify a transition near $\mu_c \approx -0.8$: for $\mu > \mu_c$ the most unstable mode is axisymmetric, whereas for $\mu < \mu_c$ the most unstable mode is non-axisymmetric.

Near criticality, slowly varying traveling waves are governed by a time-independent Ginzburg–Landau equation. The nonlinear coefficient changes sign near $\hat{\mu}_c \approx -0.65$, yielding a supercritical regime for $\mu > \hat{\mu}_c$ and a subcritical regime for $\mu_c < \mu < \hat{\mu}_c$. In the subcritical range, we classify small-amplitude steady states, including Taylor vortex flows, wavy vortices, a two-parameter family of quasi-periodic flows, and a localized traveling perturbation of the Couette flow.



\section{Introduction}


In this article, we study the  flow of a viscous fluid between two rotating cylinders of nearly equal radii.
More precisely, we consider solutions $u(t,x,y,z)$ of the Navier-Stokes equations
\begin{equation} \label{NS1}
\partial_t u + (u \cdot \nabla) u - \nu \Delta u + \nabla p = 0,
\end{equation}
\begin{equation} \label{NS2}
\nabla \cdot u = 0
\end{equation}
in the domain $\Omega$, which is defined in cylindrical coordinates $(r,\varphi,z)$ by
$$
\Omega = \Bigl\{ (r,\varphi,z) \quad  | \quad 
r_i < r < r_o \Bigr\}.
$$
In these equations, $u(t,x,y,z)$ denotes the velocity of the fluid, $p(t,x,y,z)$ the pressure, and $\nu > 0$ the viscosity. 
The domain $\Omega$ is the region between two coaxial cylinders with
the inner radius $r_i$ and the outer radius $r_o$.

The inner cylinder and the outer cylinder rotate at angular velocities $\omega_i$ and  $\omega_o$, respectively. We assume that the fluid ``sticks" to the boundary, 
namely that the velocity of the fluid equals the velocity of the cylinders on $r = r_i$ and $r = r_o$.

We define the gap $d$, the radius ratio $\eta$, and the ratio of rotation rates $\mu$ by
$$
d = r_o - r_i, \quad \eta = \frac{r_i}{r_o}, \quad \mu = \frac{\omega_o}{\omega_i}.
$$
In this article, we study the case where 
$$
\mu_c \le \mu < 1,
$$
where $\mu_c\approx -0.8$ is defined below, and the case $\mu=1$ was treated in \cite{BGGIY}.
The case $\mu <0$ corresponds
to counter-rotating cylinders.
We work in a rotating frame, with angular velocity $\Omega_{rf}$ defined by
$$
\Omega_{rf} = \frac{\omega_i + \omega_o}{2} =  \omega_i \frac{1+\mu}{2},
$$
and study the stability and bifurcations of the Couette flow in this rotating frame, given by
\begin{equation} \label{Couette}
U(r) = r \, \Omega_{rot} (r) := (A - \Omega_{rf})r + \frac{B}{r},
\end{equation}
where $A$ and $B$ are constants such that the fluid velocity matches the cylinder velocities at the inner and outer boundaries:
$$
A = \omega_i \frac{\mu - \eta^2}{1 - \eta^2},
\quad 
B = \omega_i r_i^2 \frac{1 - \mu}{1 - \eta^2}.
$$
We define the rescaled rotation rate $\widehat{\omega}$, the Reynolds number $\mathfrak{R}$, and the Taylor number $T$ by
\begin{eqnarray*}
\widehat{\omega } &=&\omega _{i}\frac{d^{2}}{\nu } =\frac{\omega _{i}r_{o}^{2}%
}{\nu }(1-\eta )^{2}, \\
\mathfrak{R} &=&\frac{\widehat{\omega }(1-\mu )}{1-\eta }=\frac{\omega
_{i}r_{o}^{2}(1-\mu )}{\nu }(1-\eta ), \\
T &=&2\widehat{\omega }\mathfrak{R=}\frac{2\omega _{i}^{2}r_{o}^{4}(1-\mu )}{%
\nu ^{2}}(1-\eta )^{3}.
\end{eqnarray*}%
Assume that
\[
1-\eta =O \Bigl(\frac{\nu }{\omega _{i}r_{o}^{2}} \Bigr)^{2/3}
\]%
then, as $\nu /(\omega _{i}r_{o}^{2}) \rightarrow 0,$ we have%
\begin{equation} \label{scaling}
(1-\eta)\rightarrow 0,\qquad \widehat{\omega }\rightarrow 0, 
\qquad \mathfrak{R} = O \Bigl((1-\eta)^{-1/2} \Bigr)\rightarrow \infty, \qquad  T=O(1).
\end{equation}
This paper is devoted to the study of the stability and bifurcations of the Couette flow (\ref{Couette}) for the
rotating Navier-Stokes equations under the scaling (\ref{scaling}).
Note that the case $\mu = 1$ was studied in \cite{BGGIY}, and the small gap situation with $\mathfrak{R} = O \Bigl((1-\eta)^{-1/2} \Bigr)$ was also studied by Nagata in \cite{Nagata24}.

\medskip

In section \ref{smallgap}, we study  the formal limit of the rotating Navier Stokes equations under the scaling \eqref{scaling}. We start with the cylindrical coordinates $(r,\varphi,z)$ and define the rescaled variables
$\hat x$, $\hat y$, and $\hat z$ by
\begin{equation} \label{rescaling}
\hat{x}=\frac{2r-r_o-r_i}{2d},\qquad \hat{y}=\frac{r_o+r_i}{2d}\varphi,\qquad \hat{z}=\frac{z}{d}
\end{equation}
and then suppress the hats. The period in $y$ is 
$$ \pi \frac{r_o + r_i}{r_o - r_i} = \pi \frac{1+\eta}{1-\eta},$$
which goes to infinity as $\eta$ goes to $1$.
Thus, in the limit \eqref{scaling}, functions of $y \in \mathbb{R}$ are no longer periodic. 

Let us consider the Fourier transform in $z$, with dual variable $\alpha$, 
and the Fourier transform in $y$, with dual variable $\beta$. We focus on solutions that vary very slowly in $y$, on length scales of order $\mathfrak{R}$.
In other words, we are interested in regimes where 
$$
\mathfrak{B} = \beta \mathfrak{R}
$$
remains of order $1$. For such solutions, the operator $\mathfrak{R} \partial_y$ remains bounded in the limiting process.

As $\eta$ goes to $1$ under the scaling (\ref{scaling}), the limiting system is
\begin{eqnarray}
(\partial _{t}-\Delta _{\bot })u_{\bot }+\nabla _{\bot }p &=&\mathfrak{R}%
x\partial _{y}u_{\bot }+T g(x) (%
\widehat{u}_{y},0)^{t}\nonumber
\\&&-(u_{\bot }\cdot \nabla _{\bot })u_{\bot } 
+\frac{T}{2}(1-\mu)(\widehat{u}_y^2,0)^{t}\label{nl1}
\\
(\partial _{t}-\Delta _{\bot })\widehat{u}_{y} &=&\mathfrak{R}x\partial _{y}%
\widehat{u}_{y}+u_{x}-(u_{\bot }\cdot \nabla _{\bot })\widehat{u}_{y}, \label{nl2}
\\
\nabla _{\bot }\cdot u_{\bot }+\mathfrak{R}\partial _{y}\widehat{u}_{y} &=&0, \label{nl3}
\end{eqnarray}%
where the subscript $\bot $ denotes the  components in the $(x,z)$ plane, $\widehat{u}_{y}=\mathfrak{R}u_{y}$, and
$$
g(x)=\frac{1+\mu}{2}-(1-\mu)x.
$$
In this system, we keep the linear terms with $\mathfrak{R} \partial_y$ 
since this operator remains bounded in the limiting regime described above.

\medskip

In Section \ref{axi}, we numerically study the stability of the Couette flow for the limiting Navier-Stokes system
under axi-symmetric perturbations, which corresponds to $\mathfrak{B} =0$. We observe the existence of a critical Taylor number $T_c(\mu,0)$, which is decreasing in $\mu$,
together with a corresponding critical vertical wavenumber $\alpha_c(\mu,0)$, which is also decreasing in $\mu$.
In particular,  for $\mu = 1$, we recover the classical value $T_c(1,0) \approx 1707$.
The Couette flow (\ref{Couette}) is linearly stable with respect to axi-symmetric perturbations
if $T < T_c(\mu,0)$, and linearly unstable if $T > T_c(\mu,0)$.

\medskip

In Section \ref{nonaxi}, we study the stability of the Couette flow for the limiting Navier-Stokes system under general non-axisymmetric perturbations.
We numerically compute the critical Taylor number $T_c(\mu,\mathfrak{B})$ as a function of $\mu$ and $\mathfrak{B}$, and obtain the following conclusions:

\begin{itemize}

\item For $\mu > \mu_c \approx -0.8$, the curve $T_c(\mu,\mathfrak{B})$ is convex in $\mathfrak{B}$. As a result, the most unstable mode is axi-symmetric, and
$$
T_c(\mu,0) = \min_{\mathfrak{B}} T_c(\mu,\mathfrak{B}).
$$

\item For $\mu < \mu_c$, the curve $T_c(\mu,\mathfrak{B})$ is concave for small $\mathfrak{B}$ and convex
for larger $\mathfrak{B}$. As a result, the most unstable mode is no longer axi-symmetric, and
$$
T_c(\mu,0) > \min_{\mathfrak{B}} T_c(\mu,\mathfrak{B}).
$$
\end{itemize}

Section \ref{GL1} is devoted to the description of small-amplitude stationary solutions of the limiting system that bifurcate, in a suitable rotating frame, when $\mu > \mu_c$ is fixed and $T$ is close to $T_c(\mu,0)$.
As $\mu > \mu_c$, if $T$ is slightly larger than $T_c(\mu)$, there exists a small interval of unstable azimuthal wavenumbers, 
$(- \mathfrak{B}(\mu),+ \mathfrak{B}(\mu))$, 
where $\mathfrak{B}(\mu)$ is defined by
$$
T_c(\mu, \mathfrak{B}(\mu)) = T.
$$
In this context, small-amplitude solutions of the limiting equation that vary slowly in $y$ can be formally described by a Ginzburg–Landau equation for the amplitude $A$, of the form
\begin{equation} \label{GLeq1}
{\partial A \over \partial t} = \alpha_1 A + \alpha_2 {\partial A \over \partial y} + \alpha_3 {\partial^2 A \over \partial y^2}
- c A | A|^2,
\end{equation}
where $\alpha_1$, $\alpha_2$, $\alpha_3$ are real numbers that vanish at criticality. Moreover,  $\alpha_1>0$ for $T>T_c(\mu)$, which is justified numerically in Section \ref{nonaxi}. The coefficient $c \in \mathbb{R}$ depends on the nonlinearity and is  computed in Section \ref{GL1}.
We restrict ourselves to traveling-wave solutions. After a change of the azimuthal coordinate, this leads to the time-independent Ginzburg–Landau equation
\begin{equation} \label{GLeq2}
 \alpha_1 A  + \alpha_3 {\partial^2 A \over \partial y^2}
- c A | A|^2 = 0.
\end{equation}
For time-dependent solutions, the rigorous mathematical justification of the amplitude equation (\ref{GLeq1}) traces back to works such as those of G.~Schneider \cite{Schn}. However, since we are specifically interested in steady solutions, after an appropriate change of variables, a sufficient justification can be found in \cite{Iooss2} using spatial dynamics. In this framework, the coordinate $y$ plays the role of time, and we seek solutions that remain small and bounded for all $y \in \mathbb{R}$.

The analysis of (\ref{GLeq2}) depends on the signs of $\alpha_1$ and $c$. Numerically, we find that $c > 0$ if $\mu > \hat \mu_c \approx -0.65$, and $c < 0$ if $\mu < \hat \mu_c$.

When $\mu > \hat \mu_c$,  we have $c > 0$, $\alpha_3<0$, so the bifurcation is supercritical. This is the same as in the special case $\mu = 1$, which has been detailed in \cite{BGGIY}. 
Therefore, for $\mu > \hat \mu_c$, the results of \cite{BGGIY} apply directly, and we do not repeat them here.

In Section \ref{GL1}, we study the regime $\mu_c < \mu < \hat \mu_c$.
Since $\alpha_3$ remains negative in this case, the situation differs from that in \cite{BGGIY}, and the bifurcation is now subcritical.
Theorem \ref{thm:case1} details the various stationary solutions of the stationary Ginzburg Landau equation (\ref{GLeq2}) that exist for $T$ close to $T_c(\mu,0)$. Besides the Couette solution $A = 0$, there exists
Taylor Vortex Flows (TVF), for which $A$ is constant, and wavy vortices (WV), for which $A$ is periodic and takes a form of plane wave in $y$.
There are also more complex quasi-periodic solutions. For $T<T_c(\mu,0)$, there exist solutions that resemble the Couette flow with a superposed localized perturbation, periodic in $z$, and traveling in the azimuthal direction.

According to the spatial dynamics theory, each of these solutions of the reduced equation (\ref{GLeq2})
corresponds to a genuine solution of our limiting Navier-Stokes system (\ref{nl1},\ref{nl2},\ref{nl3}).

Section \ref{GL2} is devoted to the regime in which $\mu$ is close to $\mu_c$ and $T$ is close to $T_c(\mu_c,0)$.
This leads to two bifurcation parameters $\tau$ and $\sigma$, where $\tau>0$ for $T>T_c(\mu_c,0)$ 
and $\sigma=0$ for $\mu=\mu_c$.
The time-independent Ginzburg-Landau equation is then of fourth order and takes the form
\begin{equation} \label{GLL}
\partial_y^4 A = \tau A + \sigma \partial_y^2 A - c A |A|^2
\end{equation}
with $c <  0$.
Theorem \ref{thm:case2} describes the stationary solutions of (\ref{GLL}) with constant
phase in various regions of the $(\tau,\sigma)$ plane. 

First, when $\sigma>0$ and $|\tau|/\sigma^2 \ll 1$,
we obtain the natural limit case ($\alpha_3 \rightarrow 0^{-}$) corresponding to Theorem \ref{thm:case1}. 

Second, when $\sigma<0$, we first study the case $|\tau|/\sigma^2 \ll 1$. In addition to the subcritically bifurcating TVF, 
there exists
\begin{itemize}
\item  a two-parameter family of quasi-periodic flows;
\item a one-parameter family of ``tubes" of heteroclinic orbits connecting a periodic flow in $y$ to itself, 
shifted by half an axial period. 
These solutions exist as long as the diameter of the limiting periodic solution is not too small; see Theorem \ref{thm:case2} for a precise statement.
\end{itemize}

Third, we study the case $|\tau/\sigma^2+1/4| \ll 1$. In this region, we have a subcritical bifurcation, leading to
quasi-periodic flows. Moreover, for $\tau<-\sigma^2/4$, there exists bifurcating homoclinic flows connecting the Couette flow to itself. These homoclinic solutions are oscillatory in $y$, with amplitudes that vanish as $y \to \pm \infty$.
To an observer, they appear as localized traveling perturbations superposed on the Couette flow.

\medskip

Note that all these solutions are periodic in $z$.


\section{The small gap approximation \label{smallgap}}


Working in cylindrical coordinates, we consider a perturbation $(u_r, u_\varphi, u_z, p)$, not necessarily axisymmetric, to the Couette flow defined in \eqref{Couette}. This perturbation satisfies (see \cite{BGGIY,Nagata23}):
\begin{align*}
   \Big( \frac{\partial}{\partial t} - \nu (\Delta_{pol} - &\frac{1}{r^2}) \Big) u_r 
   + \nu \frac{2}{r^2} \frac{\partial u_\varphi}{\partial \varphi} 
   + \frac{\partial p}{\partial r} 
   \\
   &= - \Omega_{rot} \frac{\partial u_r}{\partial \varphi} + 2\Omega_{rot} u_\varphi  
   + \frac{u_\varphi^2}{r} + 2\Omega_{rf} u_\varphi - (u \cdot \nabla) u_r,
   \\
   \Big( \frac{\partial}{\partial t} - \nu (\Delta_{pol} &- \frac{1}{r^2}) \Big) u_\varphi 
   - \nu \frac{2}{r^2} \frac{\partial u_r}{\partial \varphi} + \frac{1}{r} \frac{\partial p}{\partial \varphi} 
   \\
   &= - \Omega_{rot} \frac{\partial u_\varphi}{\partial \varphi} { - U' u_r - \Omega_{rot} u_r} 
   - \frac{u_r u_\varphi}{r} - 2\Omega_{rf} u_r  - (u \cdot \nabla) u_\varphi,
   \\
\Big( \frac{\partial}{\partial t} - &\nu \Delta_{pol}  \Big) u_z 
+ \frac{\partial p}{\partial z} 
= - \Omega_{rot} \frac{\partial u_z}{\partial \varphi}  - (u \cdot \nabla) u_z,
\end{align*}
together with the incompressibility condition in polar coordinate
$$
\frac{\partial u_r}{\partial r} + \frac{u_r}{r} + \frac{1}{r} \frac{\partial u_\varphi}{\partial \varphi} 
+ \frac{\partial u_z}{\partial z} = 0,
$$
Coriolis terms appear on the right-hand side, and $\Delta_{pol}$ denotes the Laplacian in cylindrical coordinates
$$
\Delta_{pol} = \frac{\partial^2}{\partial r^2} + \frac{1}{r} \frac{\partial}{\partial r} 
+ \frac{1}{r^2} \frac{\partial^2}{\partial \varphi^2} + \frac{\partial^2}{\partial z^2}.
$$
Rescaling as in \cite{BGGIY}, according to (\ref{rescaling}), we obtain
\begin{align*}
   \Big( \frac{\partial}{\partial  {t}} -  \Delta \Big)  {u}_{ {x}} 
   + \frac{\partial   p}{\partial  {x}} 
   =& \mathfrak{R} {x} \frac{\partial  {u}_{ {x}}}{\partial  {y}} 
   +  {\omega}(1 + \mu)  {u}_{ {y}} - ( {u} \cdot \nabla)  {u}_{ {x}} 
   +O \Bigl( {\omega} (1 - \mu)\frac{\partial  {u}_{ {x}}}{\partial  {y}} \Bigr)
   \\
   &\quad + O(1 - \eta)^2  {u}_{ {x}}
  - 2\Big( {\omega} (1 - \mu)   {x} + O(1 - \eta) \Big) {u}_{ {y}} 
  \\
  &\quad +  \Big((1 - \eta)^2 x + \frac{(1-\eta^2)}{2} \Big) {u}_{y}^2
   +  O(1 - \eta) \Big( \frac{\partial  {u}_{ {y}}}{ \partial  {y}} \Bigr),
 \\
   \Big( \frac{\partial}{\partial  {t}} -  \Delta \Big)  {u}_{ {y}} 
   + \frac{\partial   p}{\partial  {y}} 
=& \mathfrak{R} {x}\frac{\partial  {u}_{ {y}}}{\partial  {y}} + 
 {\omega} \frac{(1 - \mu)(1 + \eta^2)}{1 - \eta^2}  {u}_{ {x}} -  {\omega}(1 + \mu)  {u}_{ {x}} - ( {u} \cdot \nabla)  {u}_{ {y}}
  \\
   & 
+  O(1 - \eta) \Big(  {u}_{ {x}}  {u}_{ {y}} + \frac{\partial  {u}_{ {x}}}{ \partial  {y}} \Bigr) 
+ O \Bigl( {\omega} (1 - \mu)\frac{\partial  {u}_{ {y}}}{\partial  {y}} \Bigr)+ O(1 - \eta)^2  {u}_{ {y}},
\\
   \Big( \frac{\partial}{\partial  {t}} -  \Delta \Big)  {u}_{ {z}} 
   + \frac{\partial   p}{\partial  {z}} 
   =& \mathfrak{R} {x}\frac{\partial  {u}_{ {z}}}{\partial  {y}}  - ( {u} \cdot \nabla)  {u}_{ {z}} 
   +O \Bigl(  {\omega} (1 - \mu)\frac{\partial  {u}_{ {z}}}{\partial  {y}} \Bigr) ,
\end{align*}
$$
    \frac{\partial  {u}_{ {x}}}{\partial  {x}} 
    +\frac{\partial  {u}_{ {y}}}{\partial  {y}} + \frac{\partial  {u}_{ {z}}}{\partial  {z}} 
    = O(1-\eta) \Bigl(  {u}_{ {x}} + \frac{\partial  {u}_{ {y}}}{ \partial  {y}} \Bigr),\nonumber
$$
where we notice that the term $r^{-1} u_{\varphi}^2$ in the original system becomes, after scaling, $ \Big((1 - \eta)^2 x + \frac{(1-\eta^2)}{2} \Big) {u}_{y}^2$.
We then rescale $u_y$ in 
\[
u_{y}=\mathfrak{R}\widehat{u}_{y},
\]%
and neglect higher-order terms  (such as $\mathfrak{R}^{-1}\partial _{y}p$, terms of order 
$O(1-\eta )$, terms of order $O(\omega )$). This yields the simplified system
\begin{eqnarray*}
(\partial _{t}-\Delta )u_{x}+\partial _{x}p &=&\mathfrak{R}x\partial
_{y}u_{x}+Tg(x) \widehat{u}%
_{y}-(u\cdot \nabla )u_{x} +\frac{T}{2}(1-\mu)(\widehat{u}_y)^2\\
(\partial _{t}-\Delta )\widehat{u}_{y} &=&\mathfrak{R}x\partial _{y}\widehat{%
u}_{y}+u_{x}-(u\cdot \nabla )\widehat{u}_{y} \\
(\partial _{t}-\Delta )u_{z}+\partial _{z}p &=&\mathfrak{R}x\partial
_{y}u_{z}-(u\cdot \nabla )u_{z}, \\
\partial _{x}u_{x}+\partial _{z}u_{z}+\mathfrak{R}\partial _{y}\widehat{u}%
_{y} &=&0,
\end{eqnarray*}%
where
$$
g(x) = {1 + \mu \over 2} - (1 - \mu) x.
$$
Note that when $\mu = 1$, $g(x) = 1$ whereas when $\mu = -1$, $g(x) = - 2x$.

We now consider solutions that vary slowly in $y$. For such solutions, the second-order derivatives
$\partial _{y}^{2}$ may be suppressed in linear terms, while terms involving $\mathfrak{R}\partial _{y}$ must be retained. In the nonlinear terms, however,  terms involving $\mathfrak{R}\partial _{y}$ can be neglected. This leads to (\ref{nl1},\ref{nl2},\ref{nl3}).

Below, we give new solutions of the limiting Navier-Stokes system (\ref{nl1},\ref{nl2},\ref{nl3}); see Theorems \ref{thm:case1} and \ref{thm:case2}.


\section{Axi-symmetric instabilities \label{axi}}


We begin by studying the stability of the Couette flow for the system (\ref{nl1},\ref{nl2},\ref{nl3}) with respect to axi-symmetric perturbations; that is, perturbations that are independent of $y$. This leads to analyzing the eigenvalues $\lambda$ of the following linear system:
\begin{eqnarray}
\lambda u_{x} &=&(D^{2}-\alpha ^{2})u_{x}-Dp+ T g(x) \widehat{u}_y,  \label{eee1} \\
\lambda \widehat{u}_{y} &=&(D^{2}-\alpha ^{2})\widehat{u}_{y}+u_x,  \label{eee2} \\
\lambda u_{z} &=&(D^{2}-\alpha ^{2})u_{z}-i\alpha p,  \label{eee3} \\
0 &=&Du_{x}+i\alpha u_{z}, \label{eee4}
\end{eqnarray}
subject to the boundary conditions $u_x = D u_x = \widehat{u}_y = 0$ at $x = \pm 1/2$.

Combining (\ref{eee4}) and (\ref{eee3}), we obtain
\[
(\lambda+\alpha^2 - D^2) D u_x + \alpha^2 p = 0.
\]
Applying the operator $D$ to this identity and combining the result with (\ref{eee1}) yields
\begin{eqnarray}
(\lambda +\alpha ^{2}-D^{2})(\alpha^{2} - D^{2} )u_{x} &=& \alpha ^{2}T g(x) \widehat{u}_{y},  \label{linear1} \\
(\lambda +\alpha ^{2}-D^{2})\widehat{u}_{y} &=&u_{x}. \label{linear2}
\end{eqnarray}
For $\mu = 1$, it was proved in \cite{BGGIY} that for any $T$, 
the corresponding linear operator is self-adjoint with respect to an appropriately chosen scalar product. This property no longer holds when $\mu \ne 1$ due to the $x$-dependence of $g(x)$.

For $T = 0$, however, the linear operator is self-adjoint even when $\mu < 1$, 
and all its eigenvalues $\lambda_n$ are real and negative. Let $\lambda_0$ denote the largest , that is, the least negative, eigenvalue. These eigenvalues depend smoothly on $T$ as long as they remain distinct. Since the system's coefficients are real, an eigenvalue $\lambda_n(T)$ can become complex only if it ``collides'' with another real eigenvalue. Consequently, if no such collision occurs prior to the onset of instability, then the instability arises precisely when $\lambda_0(T)$, which remains real, crosses zero.

It is therefore reasonable to conjecture that, for values of $\mu$ sufficiently close to $1$, a zero eigenvalue appears at criticality. Numerical simulations confirm this conjecture: for $\mu > \mu_c \approx -0.8$, the first eigenvalue crosses the imaginary axis at zero.

Figure \ref{Tcmu} shows the critical Rayleigh number $T_c(\mu, 0)$ and the corresponding wave-number $\alpha_c$ as functions of $\mu$ for axi-symmetric fields. We observe that $T_c(\mu,0)$ decreases monotonically with $\mu$, reaching its minimum at $\mu = 1$ with $T_c \approx 1707$. Meanwhile, $\alpha_c$ also decreases with $\mu$, varying from approximately $4$ at $\mu = -1$ to about $3.1$ at $\mu = 1$.

\begin{figure}[th]
\begin{center}
\includegraphics[width=6cm]{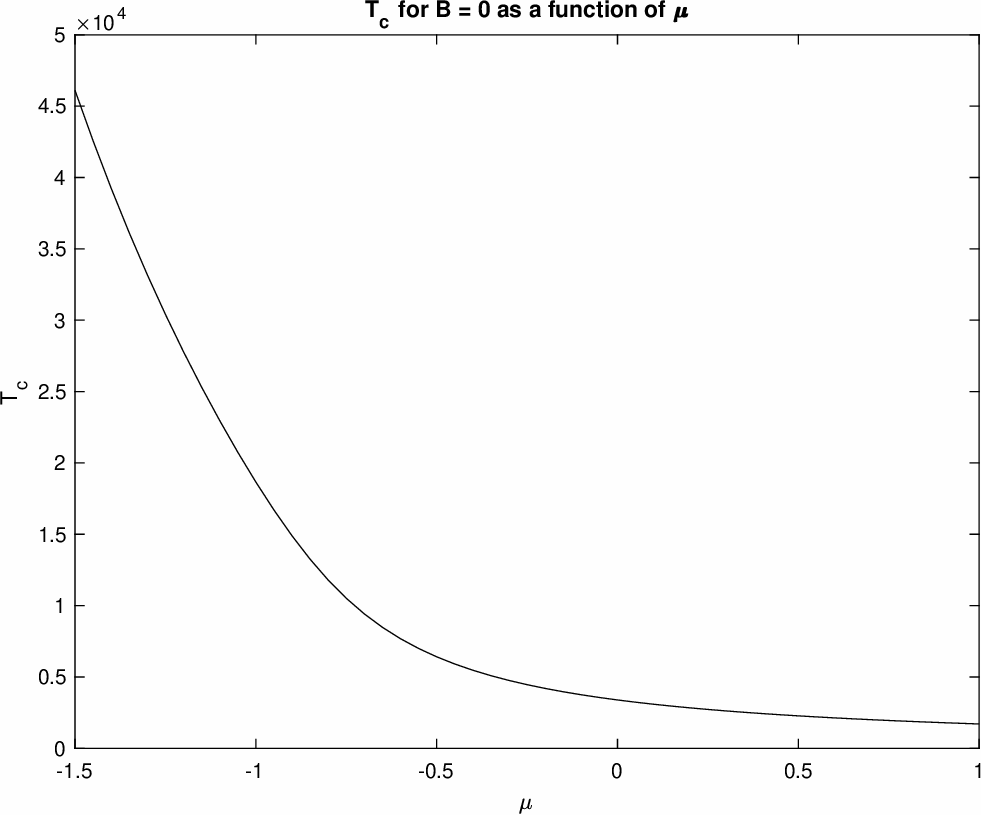}
\includegraphics[width=6cm]{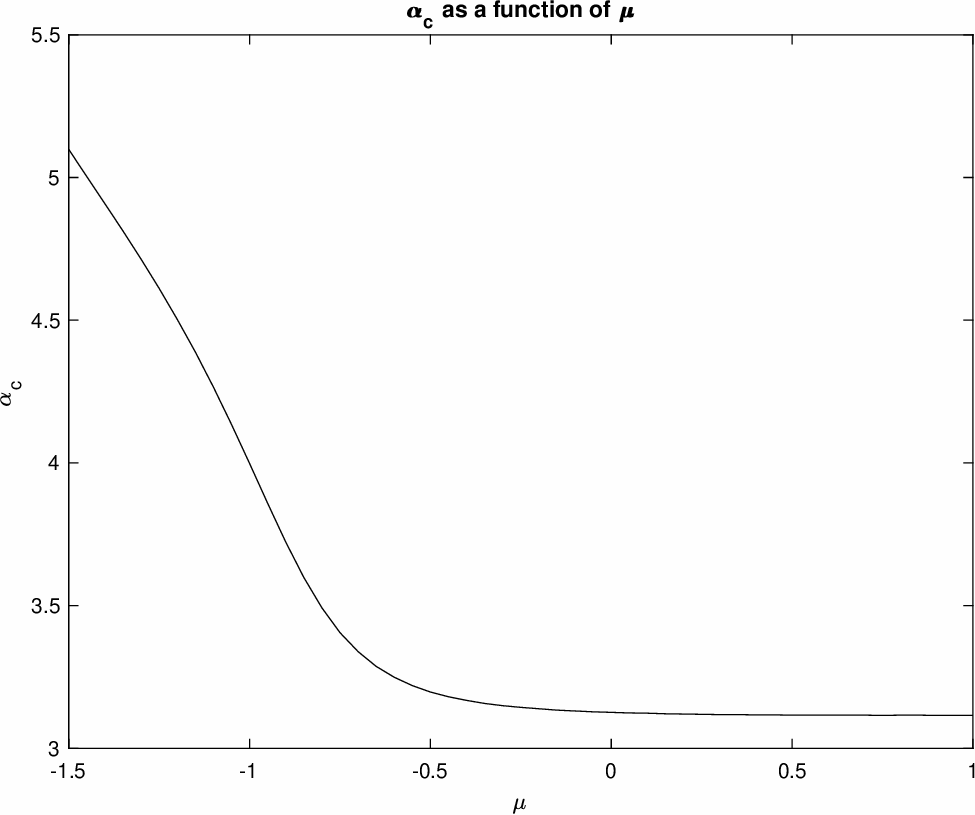}
\end{center}
\caption{$T_c$ (left) and $\alpha_c$ (right) as functions of $\mu$ for axi-symmetric fields.}
\label{Tcmu}
\end{figure}

Numerical verification further gives the expansion
\[
\tau = T - T_{c} = a(\mu) (\alpha ^{2} - \alpha _{c}^2)^{2} + O \Bigl( (\alpha^2 - \alpha_c^2)^3 \Bigr),
\]
with $a(\mu) > 0$ (the corresponding numerical simulations are not shown here).

\begin{remark}
    Note that in the standard Couette--Taylor problem with counter-rotating cylinders, the critical state near $T=0$ also corresponds to a zero eigenvalue. This configuration allows for a \emph{subcritical bifurcation} of TVF (see \cite{Iooss}, bottom of p.41, Figure III.2, for the parameters $\Omega\,(=\mu)=-0.65$ and $\eta=0.95$).
\end{remark}


\section{Non axi-symmetric instabilities \label{nonaxi}}


Let us now study non axi-symmetric perturbations, of the form $(u_x, u_y, u_z) e^{i(\alpha z+\beta y)}$.
Then the eigenvalues $\lambda$ of the linearized system satisfy
\begin{eqnarray*}
(\lambda +\alpha ^{2}-D^{2}-i\mathfrak{B}x)u_{x}+Dp &=&T g(x) \widehat{u}_{y} \\
(\lambda +\alpha ^{2}-D^{2}-i\mathfrak{B}x)\widehat{u}_{y} &=&u_{x}\\
(\lambda +\alpha ^{2}-D^{2}-i\mathfrak{B}x)u_{z}+i\alpha p &=&0 \\
Du_{x}+i\mathfrak{B}\widehat{u}_{y}+i\alpha u_{z} &=&0,
\end{eqnarray*}%
where $\mathfrak{B}=\beta\mathfrak{R}$.
This leads to 
\begin{eqnarray}
(\lambda +\alpha ^{2}-D^{2}-i\mathfrak{B}x)(\alpha^2 - D^{2})u_{x}
&=&\alpha ^{2}Tg(x) \widehat{u}_{y} \\
(\lambda +\alpha ^{2}-D^{2}-i\mathfrak{B}x)\widehat{u}_{y} &=&u_{x}, \label{new ev pb}
\end{eqnarray}%
with the boundary conditions $u_x = D u_x = u_y = 0$ at $x = \pm 1/2$.

We see that replacing $\mathfrak{B}$ by $-\mathfrak{B}$ changes $%
\lambda $ to $\overline{\lambda}$. However, the supplementary
symmetry used in \cite{BGGIY} is no longer available here due to the $x$-dependence of $g(x)$. It results that the Taylor expansion of the
critical eigenvalue at $\alpha =\alpha _{c} is now $%
\[
\lambda _{0}=ib_{1}\mathfrak{B}+a_{3}(T - T_c) +a_{4}\mathfrak{B}^{2}+ \cdots, 
\]%
where  the coefficients $b_{1},a_{3},a_{4}$ are real. 

We already know that $a_{3}>0$, by definition of $T_c$.
The sign of $a_4$ is however unknown and must be evaluated numerically.

We now study numerically $T_c(\mu,\mathfrak{B})$ as a function of $\mathfrak{B} \neq 0$. For $\mu > \mu_c \approx -0.8$,  the curve
$T_c(\mu,\mathfrak{B})$, with $\mu$ fixed, is convex  and quadratic; see figure \ref{TcB0}.
By contrast, when $\mu < \mu_c$, $T_c$ first decreases and then increases as $\mathfrak{B}$ increases.
This implies that $a_4(\mu) < 0$ if $\mu > \mu_c$ and $a_4(\mu) > 0$ if $\mu < \mu_c$.

\begin{figure}[th]
\begin{center}
\includegraphics[width=6cm]{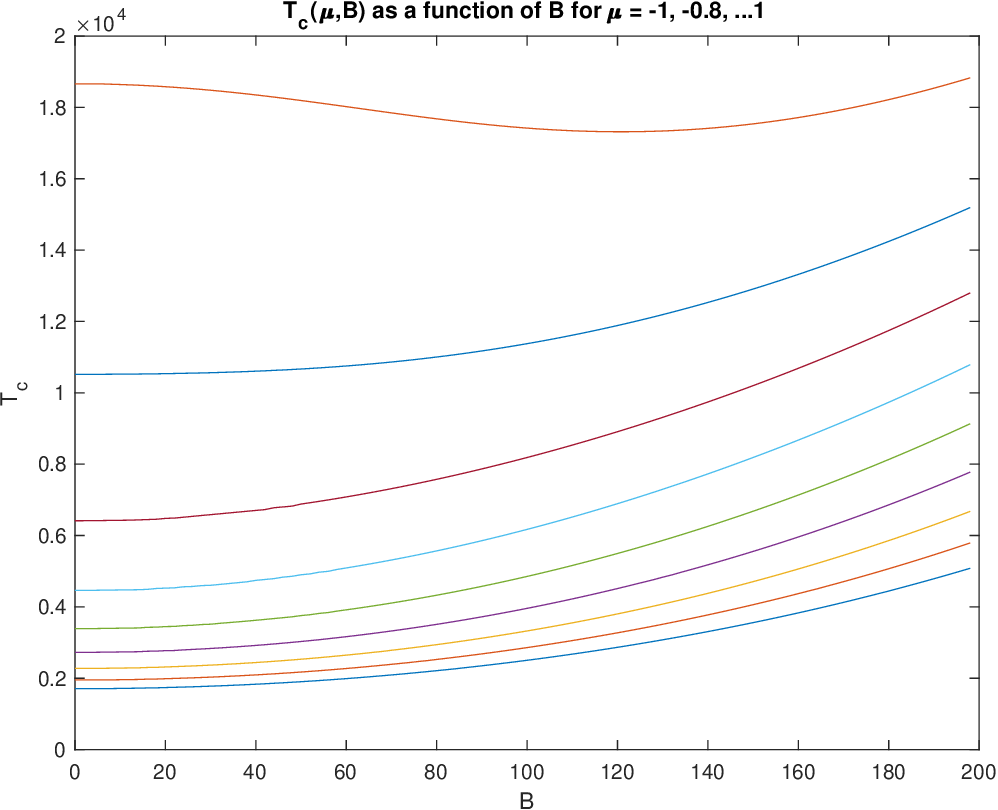}
\includegraphics[width=6cm]{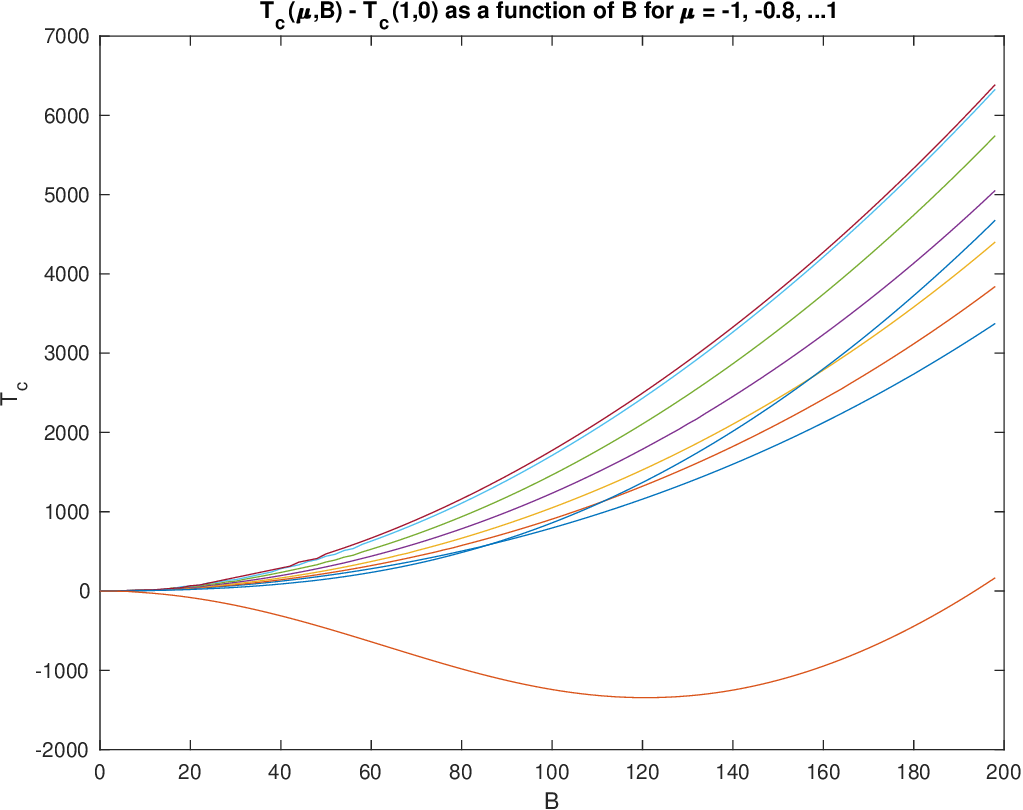}
\end{center}
\caption{Left: $T_c(\mu,\mathfrak{B})$ as a function of $\mathfrak{B}$ for $\mu = -1, -0.8, ..., 1$. The
lower curve corresponds to $\mu = 1$.
Right: $T_c(\mu,\mathfrak{B}) - T_c(\mu,0)$ as a function of $\mathfrak{B}$ for $\mu = -1, -0.8, ..., 1$. The lower curve corresponds to $\mu=-1$.
}
\label{TcB0}
\end{figure}



\section{Study of the case $\mu_c < \mu < \hat \mu_c$ \label{GL1} \label{G-L equ1}}


As discussed in the introduction, 
small-amplitude traveling-wave solutions of the limiting system~(\ref{nl1},\ref{nl2},\ref{nl3}) that vary slowly in $y$ can be described,
after choosing a suitable rotation rate for the reference frame, by a time-independent Ginzburg–Landau equation \eqref{timeGL}. This follows from the ideas of spatial dynamics theory.
In our context, this approach yields a complex second-order ordinary differential equation for the amplitude $A(y)$, whose principal part coincides with the equation for steady solutions studied below.

For $\mu > \mu_c$,  numerically we know that $a_4 < 0$, and we introduce $b_4 = - a_4 > 0$.
The time-dependent Ginzburg–Landau equation is then
\[
\frac{\partial A}{\partial t} = a_{3}\tau A + b_{1}\mathfrak{R} \frac{\partial A}{\partial y}
+ b_{4}\mathfrak{R}^{2} \frac{\partial^{2} A}{\partial y^{2}} - cA|A|^{2},
\]
where $c$ is a real number and $\tau$ the bifurcation parameter. 
The computation of $c$ is detailed in Appendix \ref{App:calcul c}.
Numerical computations show that the sign of $c$ changes for $\mu=\hat \mu_c \approx -0.65$.
In this section, we focus on $\mu$ such that  
$\mu_c < \mu < \hat \mu_c$.
In this case $c < 0$,
which corresponds to a subcritical bifurcation of the Taylor vortex flow (TVF). 
Performing the change of variables
\begin{eqnarray}
A(y,t) &=& \widetilde{A}(\widetilde{y},t), \label{changevaraiable} \\
\widetilde{y} &=& \frac{y}{\mathfrak{R}} + b_{1} t, \nonumber
\end{eqnarray}
transforms the equation into the form (dropping the tildes)
\[
\frac{\partial A}{\partial t} = a_{3}\tau A + b_{4} \frac{\partial^{2} A}{\partial y^{2}} - cA|A|^{2},
\]
which coincides with Eq.~(22) in \cite{BGGIY}.

Steady solutions correspond to traveling-wave solutions in the physical domain, 
propagating in the azimuthal direction with velocity $-b_{1}$, in addition to the angular velocity $\Omega_{rf}$ of the rotating frame. 
Time-independent solutions $A(y)$ satisfy the time-independent Ginzburg-Landau equation 
\begin{equation} \label{timeGL}
a_3 \tau A + b_4 {\partial^2 A \over \partial y^2} = c A |A|^2.
\end{equation}
The equation (\ref{timeGL}) admits the following simple solutions

\begin{itemize}
\item $A(y) = 0$, which corresponds to the Couette flow,

\item $A(y) = \rho$ with 
$c \rho^2 = a_3 \tau$, which corresponds to Taylor Vortex Flow (TVF) that bifurcates from the Couette flow at $T_c$,

\item $A(y) = \rho e^{i \beta y}$
where
\[
-c\rho ^{2}=b_{4}\beta ^{2}-a_{3}\tau.
\]%
This corresponds to wavy vortices (WV).

\end{itemize}

In this section, we prove that (\ref{timeGL}) has other solutions, which depends in a more complex way of $y$. 
  These solutions correspond to solutions of the  limiting Navier-Stokes equations (\ref{nl1},\ref{nl2},\ref{nl3}), which are traveling rotating waves, periodic in the $z$ direction.

\begin{theorem}\label{thm:case1}
    Let  $\mu_c<\mu<\hat\mu_c$, and let $\alpha = \alpha_c(\mu)$. 
    Then $c<0$ and $b_4>0$, and the  $z$-periodic Taylor vortex flow with axial period $2\pi/\alpha_c$ bifurcates subcritically at $T = T_c$.
    Let $\tau = T - T_c$ and assume that $|\tau|$ is small enough. Then
    \begin{enumerate}
        \item[(i)] There exists a two-parameter family of quasi-periodic in $y$ solutions that bifurcate. For $\tau<0$, they are either perturbations of Taylor vortices for half of the axial period, or perturbations of the full Taylor vortices. For $\tau>0$, they are perturbations of Couette flow (see Figures \ref{Fig2} and \ref{Fig4}).
        \item[(ii)] For $\tau<0$, there also exists a traveling rotating wave in the form of the Couette flow, with an added traveling localized perturbation that is periodic in $z$ (see Figure \ref{Fig2}).
    \end{enumerate}
\end{theorem}

\begin{remark}
The solutions are quasi-periodic because their amplitude is periodic whereas their phase is the sum of
a linear term and a periodic term (with the same period as the amplitude).
\end{remark}

\begin{proof}
Let us first discuss the wavy vortices and their links with the Couette flow and the Taylor vortex flow.

We note that the wavy vortices 
bifurcate from the Couette flow which (corresponding to $\rho =0$) for 
$$
\tau \leq \frac{b_{4}%
}{a_{3}}\beta ^{2};
$$ 
see Figure \ref{Fig1}. 
\begin{figure}[th]
\begin{center}
\includegraphics[width=5cm]{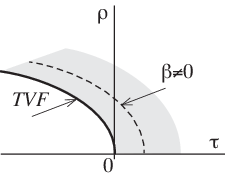}
\end{center}
\caption{Bifurcation diagram in the case $c<0$. The wavy vortices ($\beta\neq0$) are bifurcating branches parallel to the TVF branch, branching from the Couette flow at some $\tau>0$.}
\label{Fig1}
\end{figure}
The Taylor vortex flow $(TVF)$ corresponds to $\beta =0$ and is steady and
independent of $y$ for the observer, while the $(WV)$ is time periodic and
travels with a constant velocity in the azimuthal direction.




Now we study the solutions of the integrable equation
\[
a_{3}\tau A + b_{4}\frac{\partial^{2}A}{\partial y^{2}} - cA|A|^{2} = 0.
\]
Using the polar form 
\[
A = \rho e^{i\theta}.
\]
As in \cite{BGGIY}, we obtain the first integrals $H,K$ such that
\begin{align}
\rho^{2} \theta^{\prime} &= K, \label{eq:first_integral_theta} \\
\rho^{\prime 2} &= f(\rho^{2}), \label{eq:first_integral_rho}
\end{align}
where 
\[
f(X) = -\frac{K^{2}}{X} + \frac{c}{2b_{4}} X^{2} - \frac{a_{3}}{b_{4}}\tau X + H.
\]
Now $H$ and $K$ are two new parameters.

The wavy vortex $(WV)$ solutions correspond to the particular case $\theta^{\prime} = \beta$, 
$K = \beta\rho^{2}$, with $f(\rho^{2}) = 0$ and $f^{\prime}(\rho^{2}) = 0$.

We note that (\ref{eq:first_integral_theta}) is independent from (\ref{eq:first_integral_rho}), thus we may first solve (\ref{eq:first_integral_theta}) to determine $\rho(y)$ and $\rho'(y)$ and then use the knowledge of $\rho(y)$ to deduce $\theta'$ through
$$
\theta' = {K \over \rho^2}.
$$
Let us first study the particular case $K = 0$,  where $\theta$ is a constant.
The solutions of (\ref{eq:first_integral_theta}) are given by the level lines of $\rho'^2 - f(\rho^2) = 0$, which are parametrized by $H$.
The phase portraits in the cases $\tau > 0$ and $\tau < 0$ are qualitatively very different, as depicted 
 in Figure \ref{Fig2}.
\begin{figure}[th]
\begin{center}
\includegraphics[width=10cm]{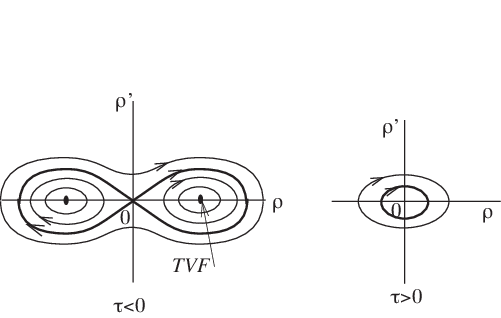}
\end{center}
\caption{Phase portraits in the case $c<0$, for $K=0$.}
\label{Fig2}
\end{figure}

For $\tau >0$ the equilibrium at $0$ corresponds to the Couette flow.
We also find solutions in which  $\rho $ is periodic in $y$ and the
argument of $A$  is independent of $y$. Due to the change of variables (\ref{changevaraiable}), 
these solutions, which are close to the Couette flow, 
are traveling waves with a periodic modulation in their amplitude.

For $\tau <0$, we can distinguish between four different types of solutions:

\begin{itemize}

\item Stationary solutions: the Couette flow ($\rho =0)$  and the TVF ($H = \frac{(a_3\tau)^2}{2cb_4} < 0$).

\item Periodic solutions which oscillates near half of the period of the TVF flow ($\frac{(a_3\tau)^2}{2cb_4} < H < 0, \rho$ has a constant sign).

\item Periodic solutions for which $\rho$ changes sign ($H > 0$).

\item Homoclinic solutions corresponding to $H = K = 0$
where $\rho (y)$ tends towards $0$ as $y\rightarrow \pm
\infty .$ \emph{This corresponds to a traveling rotating wave, resembling the Couette flow for most values of $y$, with an added traveling perturbation in the azimuthal direction, and still periodic in $z$ .}
\end{itemize}
All these solutions for $K=0$  of (\ref{timeGL}) persist under perturbation, by standard arguments developed for instance in \cite{Haragus}, as solutions of the limiting Navier-Stokes equations (\ref{nl1},\ref{nl2},\ref{nl3}).

We now turn to the case $K \neq 0$,
which leads to Figure \ref{Fig3} 
\begin{figure}[th]
\begin{center}
\includegraphics[width=11cm]{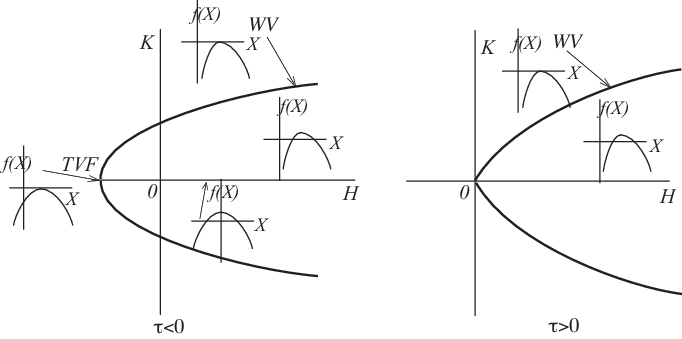}
\end{center}
\caption{$f(X)$ for $K\neq0$. Wavy vortices (WV) correspond to the limit curves.}
\label{Fig3}
\end{figure}
where we sketch the graph of $f(X)$ depending on $(H,K)$ and on the sign of $\tau$. 
This leads to distinguishing between several cases in the case $K \ne 0$:

\begin{itemize}

\item On the left of the parabola like curve, there is no small solution since the maximum of $f$ is negative.

\item On the parabola like curve, the maximum of $f$ is exactly $0$ and reached at only one point.
This corresponds to solutions with a constant $\rho$, namely to wavy vortcies.

\item On the right of the parabola like curve, $f$ has exactly two zeroes. In this case $\rho(y)$
oscillates between these two zeroes.
 The corresponding phase portrait is depicted  in Figure \ref{Fig4}.
 \begin{figure}[th]
\begin{center}
\includegraphics[width=4cm]{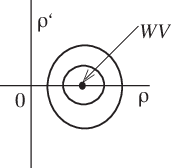}
\end{center}
\caption{Phase portrait for $K\neq0$ in the ($\rho,\rho'$) plane.}
\label{Fig4}
\end{figure}
\end{itemize}
The argument $\theta(y)$ is of the form $\theta (y)=\beta y+\phi (y)$ for some real $\beta$, where
$\phi(y)$ has the same period as $\rho(y)$.
These solutions look more exotic. In particular, they are quasi-periodic in $y$. In the original coordinates, they correspond to traveling waves whose principal part has the following form, up to a shift in $z$ and forgetting the dependency in $x$:%
\[
A(\widetilde{y})e^{i\alpha _{c}z}+c.c.,
\]%
where
\begin{align*}
A(\widetilde{y}) &= \rho(\widetilde{y})e^{i[\beta \widetilde{y}+\phi(\widetilde{y})]}, \\
\widetilde{y} &= \frac{y}{\mathfrak{R}}+b_1 t,
\end{align*}
with $\rho >0$ periodic and $\phi $ having the same period. Regarding the persistence of such solutions under perturbations, we observe that, for a suitable $\beta$, the function $A(\widetilde{y})e^{-i\beta \widetilde{y}}$ is a periodic solution of Ginzburg-Landau equation. Solutions of this type persist for the full Navier-Stokes system (\ref{nl1},\ref{nl2},\ref{nl3}) by implicit function theorem arguments; see \cite{Haragus}. This corresponds to new quasi-periodic traveling rotating waves, which bifurcate both for $\tau>0$ and for $\tau<0$.
\end{proof}


\section{Study near $\mu=\mu_c$ \label{GL2}\label{G-L equ2}}


We now examine the regime where $\mu$ is close to $\mu_c$ and $T$ is close to $T_c(\mu)$. 
Recall that the coefficient $a_4$ changes sign at $\mu_c$, while $c$ remains negative throughout.

In the vicinity of $\mu_c$ and $T_c(\mu_c)$, the critical eigenvalue $\lambda_0$ admits the Taylor expansion
\begin{equation} \label{eq:eigenvalue-expansion}
\lambda_0(\tau, \mathfrak{B}) = i b_1 \mathfrak{B} + a_3 \tau + a_4 \mathfrak{B}^2 - a_5 \mathfrak{B}^4 + \mathcal{O}(|\tau|^2 + |\mathfrak{B}|^5),
\end{equation}
where $b_1, a_3, a_4, a_5 \in \mathbb{R}$, and where we introduce the notation
\begin{align}
\tau &= T - T_c, \label{eq:tau-def}\\
a_4 &= -\sigma, \label{eq:a4-sigma}\\
\sigma &= \sigma_0(\mu - \mu_c) + \mathcal{O}(|\mu - \mu_c|^2). \label{eq:sigma-expansion}
\end{align}
Equation \eqref{eq:a4-sigma} defines the parameter $\sigma$, whose expansion near $\mu_c$ is given by \eqref{eq:sigma-expansion}.
Numerical computations yield the following sign conditions:
\begin{equation} \label{eq:sign-conditions}
a_3 > 0,\qquad a_5 > 0,\qquad b_1 \neq 0,\qquad \sigma_0 > 0.
\end{equation}
These conditions imply that in a neighborhood of the origin, the real part of $\lambda_0(\tau, 0)$ attains its minimum at $\tau = 0$. 
Furthermore, as a function of $\mathfrak{B}$, the real part of $\lambda_0(\tau, \mathfrak{B})$ exhibits a curvature that changes sign as $\mu$ varies near $\mu_c$. This sign change is governed by \eqref{eq:sigma-expansion} via relation \eqref{eq:a4-sigma}.

\begin{remark}
The problem involves two independent small parameters: $\tau = T - T_c$, which measures the deviation of the Taylor number from criticality, and $\sigma$, which is linearly related to $\mu - \mu_c$ through \eqref{eq:sigma-expansion}.
\end{remark}

The dynamics of perturbations near the Couette flow are governed by the Ginzburg-Landau equation
\begin{equation} \label{eq:basic-GL}
\partial_t A = a_3 \tau A + b_1 \mathfrak{R} \partial_y A + \sigma \mathfrak{R}^2 \partial_y^2 A - a_5 \mathfrak{R}^4 \partial_y^4 A - c A |A|^2,
\end{equation}
which depends on the two small parameters $\tau$ and $\sigma$ defined above. 
We emphasize that \eqref{eq:basic-GL} is a fourth-order equation, in contrast to the second-order equation derived in Section~\ref{G-L equ1}.

We perform the change of variables
\begin{align} \label{eq:change-variables}
\widetilde{y} &= \frac{y}{\mathfrak{R}} + b_1 t, \\
\widetilde{A}(\widetilde{y}, t) &= A(y, t).
\end{align}
Substituting \eqref{eq:change-variables} into \eqref{eq:basic-GL} and dropping the tildes for convenience yields the reduced equation
\begin{equation} \label{eq:reduced-GL}
\partial_t A = a_3 \tau A + \sigma \partial_y^2 A - a_5 \partial_y^4 A - c A |A|^2.
\end{equation}
Our primary interest lies in the steady-state solutions of \eqref{eq:reduced-GL}, i.e., functions $A(y)$ satisfying the ordinary differential equation
\begin{equation} \label{eq:steady-GL}
\partial_y^4 A = \tau A + \sigma \partial_y^2 A - c A |A|^2.
\end{equation}
Equation \eqref{eq:steady-GL} possesses several important symmetries:
\begin{itemize}
    \item Phase invariance: $A \rightarrow A e^{i h}$ for any constant $h \in \mathbb{R}$;
    \item Complex conjugation: $A \rightarrow \overline{A}$;
    \item Reversibility: $(y, A(y)) \rightarrow (-y, A(-y))$.
\end{itemize}

At the linear level, the operator $L_{\tau,\sigma}$ associated with \eqref{eq:steady-GL} has a zero eigenvalue of multiplicity eight, consisting of two identical $4 \times 4$ Jordan blocks. 
This spectral structure will be analyzed in detail below.

The objective of the following sections is to classify and characterize the steady solutions of \eqref{eq:steady-GL}. 
From a physical perspective, such solutions correspond to traveling waves that propagate at constant velocity along the azimuthal direction while remaining periodic in the axial direction. They therefore represent coherent structures near the Couette flow.

\begin{theorem}\label{thm:case2}
    For the fixed axial wave number $\alpha_c$, and for $\mu$ close to $\mu_c$ (i.e. $\sigma: = - a_4$ close to 0), we examine three different cases ,where $\tau=T-T_c$ is also close to $0$:
   
        Case 1: $\sigma>0,\qquad \tau=\varepsilon \sigma^2,\qquad|\varepsilon|\ll 1$;
        
        Case 2: $\sigma<0,\qquad \tau=\varepsilon \sigma^2,\qquad |\varepsilon|\ll 1$;

        Case 3: $\sigma<0,\qquad \sigma^2+4\tau=\varepsilon \sigma^2,\qquad |\varepsilon|\ll 1$.
\begin{itemize}
\item Case 1 extends the situation treated in Theorem \ref{thm:case1} to the regime where $\sigma$ is positive and close to $0$.
\item In Case 2, there are new types of traveling rotating waves, which can be viewed as flows on a two-parameter $(H,K)$ family of two-dimensional tori. The flow is quasi-periodic only for $(H,K)$ in a region that is locally the product of a line with a Cantor set. Moreover, there exists a one parameter family of ``tubes" of heteroclinic solutions connecting a periodic solution to itself, up to a shift by half the axial period. These solutions are asymptotically close to TVF in the azimuthal direction. They exist until the diameter of the limiting periodic solutions in the azimuthal direction becomes exponentially small as $K\rightarrow 0$.
\item In Case 3, in addition to the subcritical Taylor vortices, there exist traveling rotating wavy vortices and quasi-periodic flows of the same type as in Case 2. Moreover, for $\varepsilon<0$, there exist at least two different homoclinic solutions connecting the Couette flow to itself. These solutions exhibit damped oscillations at infinity in the azimuthal direction and localized traveling perturbations while remaining periodic in the axial direction.
\end{itemize}  
\end{theorem}

We begin by examining the linearized system associated with \eqref{eq:steady-GL}:
\begin{equation} \label{eq:linearized}
\partial_y^4 A = \tau A + \sigma \partial_y^2 A.
\end{equation}
Equation \eqref{eq:linearized} is a fourth-order complex ordinary differential equation. 
It can be recast as a first-order system of eight real ODEs. This system decouples into two identical first-order systems of four equations, one for the real part of $A$ and the other for its imaginary part. 
The corresponding matrix for each $4\times 4$ subsystem is
\begin{equation} \label{eq:L-matrix}
L_{\tau,\sigma} = \begin{pmatrix}
0 & 1 & 0 & 0 \\
0 & 0 & 1 & 0 \\
0 & 0 & 0 & 1 \\
\tau & 0 & \sigma & 0
\end{pmatrix}.
\end{equation}
The eigenvalues $\lambda$ of $L_{\tau,\sigma}$ satisfy the characteristic equation
\begin{equation} \label{eq:characteristic}
\lambda^4 - \sigma \lambda^2 - \tau = 0.
\end{equation}
Consequently, the matrix of the linearized system \eqref{eq:linearized} possesses four double eigenvalues $\lambda$, corresponding to eigenfunctions proportional to $e^{\pm i\alpha_c z}$.

The discriminant of \eqref{eq:characteristic} is
\begin{equation} \label{eq:discriminant}
\Delta = \sigma^2 + 4\tau,
\end{equation}
which vanishes on the parabola $\tau = -\sigma^2/4$, illustrated in Figure~\ref{fig:eigenvalues}. 
On this parabola, the eigenvalues of $L_{\tau,\sigma}$ are double; at the origin $(\tau,\sigma) = (0,0)$, they are even quadruply degenerate. 
The distribution of eigenvalues in the $(\tau,\sigma)$ parameter plane is depicted in Figure~\ref{fig:eigenvalues}.

\begin{figure}[h]
\centering
\includegraphics[width=6cm]{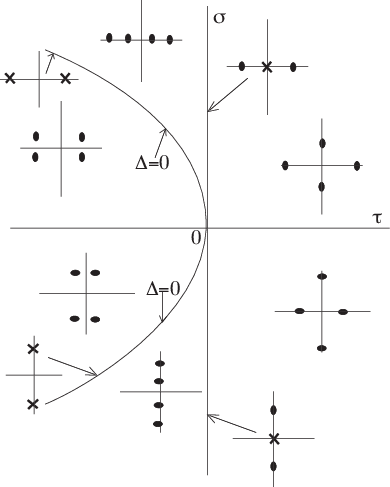}
\caption{Location of roots $\lambda$ of $\lambda^4 - \sigma \lambda^2 - \tau = 0$: dots indicate simple roots, crosses indicate double roots.}
\label{fig:eigenvalues}
\end{figure}

\medskip

In what follows, we restrict our attention to solutions where the amplitude $A$ possesses a constant phase. 
Under this assumption, we may take $A$ to be real without loss of generality. 
The system then reduces to a four-dimensional real ODE of the form
\begin{equation} \label{eq:reduced-system}
\frac{dX}{dy} = L_{\tau,\sigma} X + C(X,X,X),
\end{equation}
where $X = (x_1, x_2, x_3, x_4)^T \in \mathbb{R}^4$, and $C$ is the symmetric trilinear operator defined by
\begin{equation} \label{eq:cubic-operator}
C(X,Y,Z) = \begin{pmatrix}
0 \\ 0 \\ 0 \\ -c\, x_1 y_1 z_1
\end{pmatrix}.
\end{equation}
The system \eqref{eq:reduced-system} is reversible under the symmetry $S$ defined by
\begin{equation} \label{eq:reversibility}
S X = (x_1, -x_2, x_3, -x_4)^T.
\end{equation}
With this formulation, we have arrived at a setting that falls within the framework developed in \cite{Io95}. 
Appropriate scalings, as detailed therein, permit a reduction to normal forms that have been systematically studied in \cite{Haragus}. 
The analysis of \eqref{eq:reduced-system} will therefore proceed along the lines established in these references, adapted to the specific structure of the nonlinearity \eqref{eq:cubic-operator} and the reversibility symmetry \eqref{eq:reversibility}.


\subsection{Study in the region $\protect\tau =\protect\varepsilon \protect%
\sigma ^{2},|\protect\varepsilon |\ll 1,$ near the line $\protect\tau =0$}

The corresponding region in the parameter plane $(\tau ,\sigma )$ is the
union of two symmetric horn shaped areas centered on the $\sigma $ axis; see
Figure \ref{region1}.

\begin{figure}[h]
\begin{center}
\includegraphics[width=3cm]{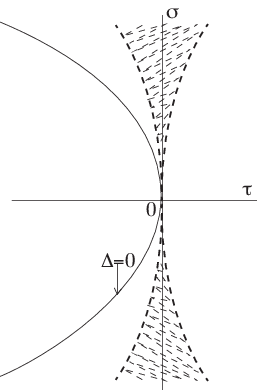}
\end{center}
\caption{Region of study $\protect\tau=\protect\varepsilon \protect\sigma^2,|%
\protect\varepsilon|\ll 1$.}
\label{region1}
\end{figure}
When $\tau =0$, the eigenvalues of $L_{0,\sigma }$ are such that%
\begin{eqnarray*}
\lambda =0\text{,} \qquad
\lambda ^{2} =\sigma .
\end{eqnarray*}%
The Jordan block for the double eigenvalue $0$ is 
\[
\left( 
\begin{array}{cc}
0 & 1 \\ 
0 & 0%
\end{array}%
\right) 
\]%
associated with the eigenvector $\xi _{0}$ and the generalized eigenvector $\xi _{1}$
defined by%
\begin{equation}\label{eigenvector}
\xi _{0}=\left( 
\begin{array}{c}
1 \\ 
0 \\ 
0 \\ 
0%
\end{array}%
\right), \qquad \xi _{1}=\left( 
\begin{array}{c}
0 \\ 
1 \\ 
0 \\ 
0%
\end{array}%
\right) , 
\end{equation}
which satisfy%
\[
S\xi _{0}=\xi _{0}, \qquad S\xi _{1}=-\xi _{1}. 
\]%
The two other eigenvalues are 
\[
\lambda _{\pm }=\pm \sqrt{\sigma } .
\]%
They are real and have opposite signs when $\sigma >0,$ and they are purely imaginary complex conjugates when $\sigma <0.$ Let us make the following scaling:%
\begin{equation}
x_{1}=\sigma \widetilde{x_{1}}, \qquad y=|\sigma |^{-1/2}\widetilde{y},
\label{scaling 1}
\end{equation}%
then dropping the tildes, (\ref{eq:reduced-system}) becomes%
\begin{equation}
\frac{d^{4}x_{1}}{dy^{4}}=\varepsilon x_{1}+\mathrm{sgn}\,\sigma \frac{%
d^{2}x_{1}}{dy^{2}}-cx_{1}^{3}.  \label{reducedsyst 1}
\end{equation}

\subsubsection{Case $\protect\sigma >0,\protect\tau =\protect\varepsilon 
\protect\sigma ^{2},0\leq |\protect\varepsilon |\ll 1.$}

Now we study the system%
\[
\frac{d^{4}x_{1}}{dy^{4}}=\varepsilon x_{1}+\frac{d^{2}x_{1}}{dy^{2}}%
-cx_{1}^{3}, 
\]
where $c<0$. 
This case corresponds to the two-dimensional study of a ``reversible
Takens-Bogdanov bifurcation $0^{2+}$". It also occurs in several physical
problems; see Section 4.1 of the book \cite{Haragus}. After the scaling, the
eigenvalues satisfy%
\[
\lambda ^{4}-\lambda ^{2}-\varepsilon =0.
\]%
For $\varepsilon =0,$ the eigenvector $\xi _{0}$ and generalized
eigenvector $\xi _{1}$ belonging to the double eigenvalue $0$ are as defined
above. The two other eigenvalues for $\varepsilon =0$ are 
\[
\lambda _{\pm }=\pm 1,
\]%
and hence not close to $0.$ It results, by a classical center manifold argument
and using the reversibility symmetry, that the dynamics in the neighborhood
of the origin are governed by a 2-dimensional system of the form (see Section
4.1 in \cite{Haragus})%
\begin{eqnarray*}
\frac{dU}{dy} &=&V, \\
\frac{dV}{dy} &=&\phi (\varepsilon ,U).
\end{eqnarray*}%
The 2-dimensional center manifold reads as%
\[
X=U\xi _{0}+V\xi _{1}+\Phi (\varepsilon ,U,V),
\]%
where%
\[
S\Phi (\varepsilon ,U,V)=\Phi (\varepsilon ,U,-V).
\]%
Since, for $\varepsilon \neq 0$ the critical eigenvalues $\lambda $ are such
that $\lambda ^{2}=-\varepsilon +\varepsilon ^{2}+O(\varepsilon ^{3}),$ we
deduce that the linear part in $U$ of $\phi $ is%
\[
\phi _{lin}(\varepsilon ,U)=(-\varepsilon +\varepsilon ^{2}+O(\varepsilon
^{3}))U.
\]%
Moreover, the function $\phi $ has no quadratic term in $U$, since the system (%
\ref{reducedsyst 1}) is odd, and the system on the center manifold preserves
this oddness. It is easy to check that the principal part of the reversible
two-dimensional system on the center manifold is now
\begin{eqnarray}
\frac{dU}{dy} &=&V , \label{Takens-Bogdanov} \\
\frac{dV}{dy} &=&-\varepsilon U+cU^{3}.  \nonumber
\end{eqnarray}%
This gives phase portraits as in Figure \ref{Fig2}, with $\tau$ replaced by $\varepsilon$. Thus, this case is a natural extension of the situation studied in Section \ref{G-L equ1}, when $\mu$ approaches $\mu_c$ from above ($\sigma>0$).

\begin{remark}
In this situation, we would be able to study the case where the phase of $A$ is
not fixed. This leads to a quadruple zero eigenvalue, while the other four eigenvalues remain far from the imaginary axis. This is exactly what is done
in Section \ref{G-L equ1}. As a result, one obtains many more exotic solutions than those of (\ref{Takens-Bogdanov}), as expressed in Theorem \ref{thm:case1}.
\end{remark}

\subsubsection{Case $\protect\sigma <0,\protect\tau =\protect\varepsilon 
\protect\sigma ^{2},0\leq |\protect\varepsilon |\ll 1.$}

Now we study the system%
\[
\frac{d^{4}x_{1}}{dy^{4}}=\varepsilon x_{1}-\frac{d^{2}x_{1}}{dy^{2}}%
-cx_{1}^{3}, 
\]
where $c<0$. 
For $\varepsilon =0,$ the eigenvalues are $0$ and $\pm i$. The eigenvectors $\xi
_{0}$ and $\xi _{1}$, associated with the double zero eigenvalue, are defined as in \eqref{eigenvector}.
The eigenvectors $\zeta _{0}$ and $\overline{\zeta _{0}}$ associated with
eigenvalues $\pm i$ are defined by%
\[
\zeta _{0}=\left( 
\begin{array}{c}
1 \\ 
i \\ 
-1 \\ 
-i%
\end{array}%
\right) . 
\]
This case corresponds to the four-dimensional study of a ``reversible $%
0^{2+}(i\omega )$ bifurcation". It occurs in several physical problems; see
Section 4.3.1 of \cite{Haragus}. In this case, the center manifold
reduction is not applicable, and one must work directly with a four-dimensional system, which can be rewritten in ``normal form"; see Section 4.3.1 of \cite{Haragus}.
This means that we can define new coordinates through the nonlinear, close-to-identity change of coordinates
\[
X=U\xi _{0}+V\xi _{1}+W\zeta _{0}+\overline{W}\overline{\zeta _{0}}+\Phi
(\varepsilon ,U,V,W,\overline{W}),
\]%
such that%
\begin{eqnarray*}
S(U,V,W,\overline{W})^{t} &=&(U,-V,\overline{W},W)^{t}, \\
S\Phi (\varepsilon ,U,V,W,\overline{W}) &=&\Phi (\varepsilon ,U,-V,\overline{%
W},W),
\end{eqnarray*}%
and such that the principal part of the system in $\mathbb{R}^{2}\times 
\mathbb{C}
$ takes the form (see Section 4.3.1 of the book \cite{Haragus} for indications on
its computation) 
\begin{eqnarray}
\frac{dU}{dy} &=&V  \nonumber \\
\frac{dV}{dy} &=&\varepsilon U-cU^{3}-6cU|W|^{2}  \label{normalform1} \\
\frac{dW}{dy} &=&i(1+\frac{\varepsilon }{2})W-\frac{3ic}{2}W(U^{2}+|W|^{2}).
\nonumber
\end{eqnarray}%
The coefficients of the normal form are computed by the standard way.
First integrals of system (\ref{normalform1}) are 
\begin{eqnarray*}
K &=&|W|^{2}, \\
H &=&V^{2}-\varepsilon U^{2}+\frac{c}{2}U^{4}+6cKU^{2}.
\end{eqnarray*}%
Indeed, from \eqref{normalform1}, we have $W' = iRW$, where $R$ is a real function. Therefore,
$$
K' = W' \overline{W} + W \overline{W}' = i RW \overline{W} - iWR\overline{W} = 0.
$$
Moreover $H' = 0$ follows from direct computations.
It follows that the four-dimensional phase portrait has the following structure: 
the projections of the solutions onto the $W$ plane are circles, while their projections onto the $(U,V)$ plane are described by 
\[
V^{2}=F_{H,K}(U)\overset{def}{=}(\varepsilon -6cK)U^{2}-\frac{c}{2}U^{4}+H.
\]%
as indicated in Figure \ref{region1a}. 

\begin{figure}[h]
\begin{center}
\includegraphics[width=8cm]{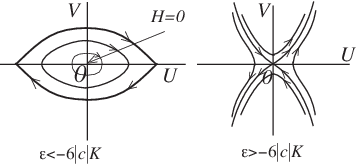}
\end{center}
\caption{Phase portraits given by $V^2=f_{H,K}(U)$. No small bounded solutions for $\varepsilon>-6|c|K$. For $\varepsilon<-6|c|K$ the two non $0$ equilibria are TVF for $K=0$.}
\label{region1a}
\end{figure}

We first observe that, for $\varepsilon>-6|c|K$, the only relevant solution in the $(U,V)$-plane is the equilibrium $U=V=0$. Through the oscillation of $W(y)$, this corresponds to periodic solutions of the first kind that remain close to the Couette flow.

For $\varepsilon <-6|c|K$, the two equilibria $$U=\pm \left|\frac{\varepsilon -6cK}{c}%
\right|^{1/2}$$ in the $(U,V)$ plane correspond to a one-parameter family of
periodic solutions $W_{K}(y)$ of a first kind. When $K=0$, this family reduces to the TVF flow, which is independent of $y$, and bifurcates subcritically. The heteroclinic orbits in the $%
(U,V)$ plane give solutions connecting one equilibrium to the other. In the full four-dimensional phase space, this gives a one-parameter family of two-dimensional heteroclinic tubes, with varying diameters due to $W(y)$, connecting a periodic solution of the first kind to the other. These two periodic solutions represent the same equilibrium, shifted in the $z$-coordinate by half of the axial period.

The closed curves in the $(U,V)$ plane correspond to solutions running on
two-dimensional tori in four-dimensional space. In general, these solutions are quasi-periodic.

Let us now discuss whether these solutions persist for the full limiting Navier-Stokes equations (\ref{nl1},\ref{nl2},\ref{nl3}), beyond the cubic normal form \eqref{normalform1}.
A general result (see Section 4.3.1 of \cite{Haragus}) states that the periodic solutions of the first kind persist for the full system. It is proved that, for $\varepsilon<-6|c|K$, the tubes of heteroclinic orbits persist provided their diameter, that is, $K$, is not too small. In fact, the same proof in \cite{Lomb} shows that the above tubes persist until their diameter becomes exponentially small as $K\rightarrow 0$. By contrast, the heteroclinic orbit to the TVF solution obtained from the normal form when $K=0$ has not been proved to persist.

For the solutions lying on invariant tori, the quasi-periodic solutions persist for $(H,K)$ lying in a region of the $(H,K)$-plane that is locally the product of a line with a Cantor set. This type of result goes back to a proof given in \cite{Io-Ki} for the same normal form. In particular, this gives quasi-periodic bifurcating solutions close to Couette flow for $\tau <0$.

\subsection{Study in the region $\protect\sigma ^{2}+4\protect\tau \ll \protect%
\sigma ^{2}$ and $\sigma < 0$ (near the bottom curve $\Delta =0$)}

The corresponding region in the parameter plane $(\tau ,\sigma )$ is a
curved horn shaped area centered on the parabola given by $\Delta =0$; see Figure \ref{region2}. 
\begin{figure}[h]
\begin{center}
\includegraphics[width=4cm]{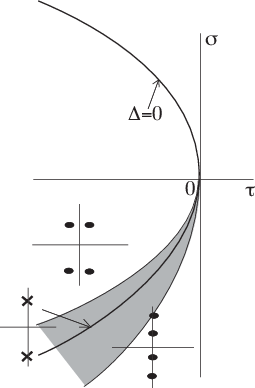}
\end{center}
\caption{Region $\sigma^2+4\tau \ll \sigma^2$ and $\sigma < 0$.}
\label{region2}
\end{figure}
Let us set 
\begin{equation*}
\sigma ^{2}+4\tau =\varepsilon \sigma ^{2},
\end{equation*}%
with $|\varepsilon |\ll 1,$ and make the scaling%
\begin{equation*}
x_{1}=\sigma ^{2}\widetilde{x_{1}},\quad y=|\sigma |^{-1/2}\widetilde{y}.
\end{equation*}%
Consequently, 
$$
x_{2}=\sigma ^{2}|\sigma|^{1/2}\widetilde{x_{2}},\quad x_{3}=\sigma ^{2}|\sigma|\widetilde{x_{3}},\quad x_{4}=\sigma ^{2}|\sigma|^{3/2}\widetilde{x_{4}}.
$$
Dropping the tildes, (\ref{eq:reduced-GL}) becomes%
\begin{equation}
\frac{dX}{dy}=L^{(0)}X+\varepsilon L^{(1)}X+\sigma ^{2}C(X,X,X),
\label{reducedsyst 2}
\end{equation}%
with 
\begin{equation*}
X=(x_{1},x_{2},x_{3},x_{4})^{t},
\end{equation*}%
\begin{equation*}
L^{(0)}=\left( 
\begin{array}{cccc}
0 & 1 & 0 & 0 \\ 
0 & 0 & 1 & 0 \\ 
0 & 0 & 0 & 1 \\ 
-1/4 & 0 & -1 & 0%
\end{array}%
\right) ,
\qquad 
L^{(1)}=\left( 
\begin{array}{cccc}
0 & 0 & 0 & 0 \\ 
0 & 0 & 0 & 0 \\ 
0 & 0 & 0 & 0 \\ 
1/4 & 0 & 0 & 0%
\end{array}%
\right) ,
\end{equation*}%
and $C$ defined by%
\begin{equation*}
C(X,Y,Z)=\left( 
\begin{array}{c}
0 \\ 
0 \\ 
0 \\ 
-cx_{1}y_{1}z_{1}%
\end{array}%
\right),
\end{equation*}%
where $c<0$.
The eigenvalues $\lambda $ of the linear operator satisfy
\begin{equation*}
(\lambda ^{2}+\frac{1}{2})^{2}-\frac{\varepsilon }{4}=0.
\end{equation*}%
For $\varepsilon=0$, the system has a pair of double imaginary eigenvalues $\pm 
 i/\sqrt{2}$. For $\varepsilon >0$ we have four symmetric
eigenvalues on the imaginary axis, whereas for $\varepsilon<0$, we have two
pairs of eigenvalues symmetric with respect to real and imaginary axis; see
Figures \ref{fig:eigenvalues} and \ref{region2}. This case corresponds to the four-dimensional study of a ``reversible $(i\omega )^{2}$ bifurcation (1-1 resonance)", which occurs in several physical problems; see Section 4.3.3 of the book \cite{Haragus}.

The eigenvectors $\zeta_{0}$ and $\overline{\zeta_{0}}$, together with the generalized eigenvectors $\zeta_{1}$ and $\overline{\zeta_{1}}$ are defined as follows. For the eigenvalue $ i/\sqrt{2}$, we have,
\begin{equation*}
\zeta _{0}=\left( 
\begin{array}{c}
1 \\ 
\frac{i}{\sqrt{2}} \\ 
-\frac{1}{2} \\ 
-\frac{i}{2\sqrt{2}}%
\end{array}%
\right) ,\text{   }\zeta _{1}=\left( 
\begin{array}{c}
0 \\ 
1 \\ 
i\sqrt{2} \\ 
-\frac{3}{2}%
\end{array}%
\right) ,
\end{equation*}%
which satisfy%
\begin{eqnarray*}
(L_{0}-i/\sqrt{2})\zeta _{0} &=&0,\text{   }(L_{0}-i/\sqrt{2})\zeta _{1}=\zeta _{0}, \\
S\zeta _{0} &=&\overline{\zeta _{0}},\text{   }S\zeta _{1}=-\overline{\zeta _{1}}.
\end{eqnarray*}%
As it is now classical (see section 4.3.3 of 
\cite{Haragus}), we can define a nonlinear change of coordinates such that%
\begin{equation*}
X=U\zeta _{0}+V\zeta _{1}+\overline{U}\overline{\zeta _{0}}+\overline{V}%
\overline{\zeta _{1}}+\Phi (\varepsilon ,U,V,\overline{U},\overline{V}),
\end{equation*}%
where $\Phi $ is smooth with a Taylor expansion of the form%
\begin{equation*}
\Phi (\varepsilon ,U,V,\overline{U},\overline{V})=\sum_{m+p+q+r+s\geq
2}\varepsilon ^{m}U^{p}\overline{U}^{q}V^{r}\overline{V}^{s}\Phi
_{pqrs}^{(m)}
\end{equation*}%
in such a way that the system satisfied by $(U,V)$ is in normal form. The
principal part of the normal form is:%
\begin{equation}\label{normal form2}
    \begin{aligned}
        \frac{dU}{dy} &=\frac{i}{\sqrt{2}}(1-\frac{\varepsilon }{8})U+V+iU[\beta
|U|^{2}+i\frac{\gamma }{2}(U\overline{V}-\overline{U}V)]  \\
\frac{dV}{dy} &=\frac{i}{\sqrt{2}}(1-\frac{\varepsilon }{8})V+iV[\beta
|U|^{2}+i\frac{\gamma }{2}(U\overline{V}-\overline{U}V)]\\
&+U[-\frac{%
\varepsilon }{8}+\frac{3}{2}\sigma ^{2}c|U|^{2}+i\frac{d}{2}(U\overline{V}-%
\overline{U}V)],
    \end{aligned}
\end{equation}
where $d=c\sigma^2/(2\sqrt{2})$,  and the coefficients $\beta$ and $\gamma$ are proportional to $\sigma^{2}c$. They can be computed easily by following the procedure described on pp. 315--319 of \cite{Haragus}. The most important one is the negative
coefficient of $U|U|^{2}$ in the second equation. Since this system is integrable,
 we know all its small bounded solutions for $\varepsilon $ close to $0.$

First integrals are 
\begin{eqnarray*}
K &=&\frac{i}{2}(U\overline{V}-\overline{U}V), \\
H &=&|V|^{2}+(\frac{\varepsilon }{8}-dK)|U|^{2}-\frac{3}{4}\sigma
^{2}c|U|^{4},
\end{eqnarray*}%
and defining $r_{0},r_{1},\theta _{0},\theta _{1}$ by%
\begin{equation*}
U=r_{0}e^{i(\frac{y}{\sqrt{2}}+\theta _{0})},\text{   }V=r_{1}e^{i(\frac{y}{\sqrt{2}}%
+\theta _{1})},
\end{equation*}%
we obtain%
\begin{eqnarray*}
K &=&r_{0}r_{1}\sin (\theta _{1}-\theta _{0}), \\
H &=&r_{1}^{2}+(\frac{\varepsilon }{8}-dK)r_{0}^{2}-\frac{3}{4}\sigma
^{2}cr_{0}^{4}.
\end{eqnarray*}%
Now we set%
\begin{equation*}
u_{0}=r_{0}^{2},\quad u_{1}=r_{1}^{2}.
\end{equation*}%
From \eqref{normal form2}, we know that $U' = V + iGU$, where $G$ is a real function.
Then,
\begin{align*}
    u_0' &= U' \overline{U} + U \overline{U}'\\
    &= (V + iGU) \overline{U} + U(\overline{V} - iG\overline{U})\\
    &= V \overline{U} + U \overline{V}\\
    &= \frac{1}{2} r_0 r_1 \cos (\theta_1 - \theta_0).
\end{align*}
Therefore, we obtain%
\begin{eqnarray*}
(\frac{du_{0}}{dy})^{2} &=&4(u_{0}u_{1}-K^{2})\overset{\mbox{def}}{=}%
4f_{H,K}(u_{0},\varepsilon ), \\
u_{1} &=&(-\frac{\varepsilon }{8}+dK)u_{0}+\frac{3}{4}\sigma
^{2}cu_{0}^{2}+H,
\end{eqnarray*}%
and for $K\neq 0$%
\begin{equation*}
\theta _{1}-\theta _{0}=-\mbox{sgn}~(\frac{du_{0}}{dy})\tan ^{-1}\left( \frac{1}{K}%
f_{H,K}^{1/2}(u_{0},\varepsilon )\right) +\theta _{\ast },
\end{equation*}%
with $\theta _{\ast }$ arbitrary. Equilibria are given by%
\begin{equation*}
f_{H,K}(u_{0},\varepsilon )=0,\quad \frac{\partial }{\partial u_{0}}%
f_{H,K}(u_{0},\varepsilon )=0.
\end{equation*}%
This leads to the system%
\begin{eqnarray}
H &=&-2\widetilde{\varepsilon }u_{0}-\frac{9}{4}\sigma ^{2}cu_{0}^{2},
\label{equil region2} \\
K^{2} &=&-u_{0}^{2}(\widetilde{\varepsilon }+\frac{3}{2}\sigma ^{2}cu_{0}), 
\notag
\end{eqnarray}%
where%
\begin{equation*}
\widetilde{\varepsilon }=-\frac{\varepsilon }{8}+dK.
\end{equation*}%
For $K=0,$ we replace the formula above by%
\begin{eqnarray*}
\theta _{1} &=&\theta _{0},\text{   }\frac{d\theta _{0}}{dy}=-\frac{\varepsilon}{8\sqrt{2}}+\beta r_{0}^{2}, \\
r_{1}^{2} &=&(r_{0}^{\prime })^{2}=-\frac{\varepsilon }{8}r_{0}^{2}+\frac{3}{%
4}\sigma ^{2}cr_{0}^{4}+H.
\end{eqnarray*}%
We plot in Figure \ref{region2-1} the curve $\Gamma$ given by (\ref{equil region2}) defined in
parametric form, with $u_{0}=r_{0}^{2}\geq 0.$ We also plot the shape of 
$$f_{H,K} = \frac{3}{4} \sigma^2 c u_0^3 + \widetilde{\varepsilon}u_0^2 + H u_0 - K^2,$$ 
for various values of $(H,K).$
\begin{figure}[th]
\begin{center}
\includegraphics[width=10cm]{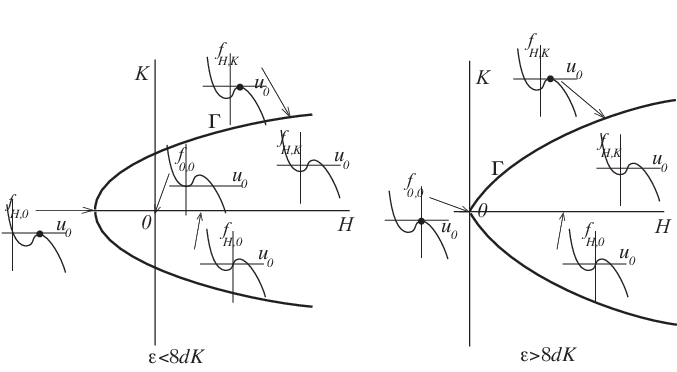}
\end{center}
\caption{Graphs of $f_{H,K}(.,\protect\varepsilon )$. Bounded orbits exist for $(H,K)$ in the 
region bounded by the curves $\Gamma$.}
\label{region2-1}
\end{figure}
The system has equilibria for $(H,K)\in \Gamma$. These correspond to equilibria in $(u_0,u'_0,\theta_1-\theta_0)$. For $\varepsilon<0$, $K=0$, and $H$ taking its minimum value $H_{min}$, the equilibrium corresponds to the TVF flow, which bifurcates subcritically.

In all cases, the region enclosed by $\Gamma$ contains periodic orbits in $u_0(y)$. In addition, when $\varepsilon<0$, $H=K=0$ is an isolated point on $\Gamma$ at which  $f_{0,0}$ has a double root at $u_0 =0$. Hence, the system has an orbit homoclinic to $0$, that is, to the Couette flow.

The projections of the bounded orbits onto the $(u_{0},u_{0}^{\prime })$ plane
for $K\neq 0,$ and onto the $(r_{0},r_{0}^{\prime })$ plane for $K=0,$ are
shown in Figure \ref{region2-2}. 

\begin{figure}[h]
\begin{center}
\includegraphics[width=10cm]{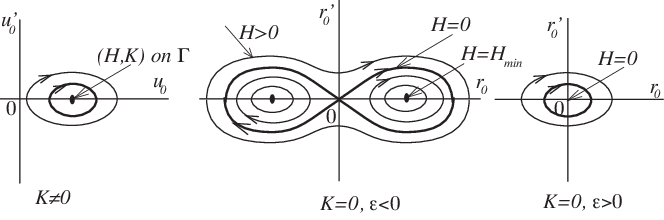}
\end{center}
\caption{Plot of the projections in the $(u_{0},u_{0}^{\prime })$ plane for $K\neq0$, and in the $(u_0,u'_0)$ plane for $K=0$, of
the bounded orbits, obtained by varying $H$ for a fixed $K$.}
\label{region2-2}
\end{figure}

For $K=0,$ $H=H_{min}$, and $\varepsilon<0$, the system has two equilibria $\pm r_0$ satisfying
\begin{equation*}
r_{0}^{2}=\frac{\varepsilon }{12\sigma ^{2}c}.
\end{equation*}
These equilibria correspond to solutions with principal part
\begin{equation*}
X=r_{0}e^{i\beta r_{0}^{2}y}\zeta _{0}+c.c.,
\end{equation*}%
and hence to wavy vortices with amplitude of order $\varepsilon^{1/2}$ in the subcritical region. These solutions are defined up to a translation in $y$, and they arise on the left-hand side of the curve $\Delta=0$ in the parameter plane $(\sigma,\tau)$. From the observer's point of view, they are periodic traveling waves that bifurcate subcritically.

In addition, for $K=0$ and $\varepsilon<0$, there are two orbits homoclinic to $0$ (the Couette flow), one obtained from the other by a shift of $\pi/\alpha_c$ in the axial direction $z$. For the normal form (\ref{normal form2}), these solutions correspond to a circle of homoclinics to the origin, due to the arbitrariness of the phase of $(U,V)$. In the physics literature, the case in which the nonlinear coefficient $c$ is negative is called ``the focusing case". It also appears in the context of ``bright solitary waves" in water-waves theory and optics. As for the persistence of these solutions, an important fact is that, among the circle of homoclinics, at least two of them persist for the full non-truncated system (see \cite{Haragus} p.223), and therefore for the Navier-Stokes system
(\ref{nl1},\ref{nl2},\ref{nl3}). These subcritical solutions have period $2\pi/\alpha_c$ in $z$ and exhibit a localized modulation in $y$. As a result, they remain periodic in $z$ while resembling the Couette flow for most values of $y$.

For $K\neq 0,$ there are equilibria corresponding to the curve $\Gamma$
for suitable values of $H.$ Then phase-space analysis is the same as that discussed in Section 5 of \cite{BGGIY}. It should be
noted that, since the coefficient $d$ in (\ref{normal form2}) is nonzero, such small bounded solutions exist on both sides of the curve $\Delta=0$ in the parameter plane $(\tau,\sigma)$.

As for the persistence of all the periodic solutions above, when one considers the full
non-truncated system and its validity for the Navier-Stokes system (\ref{nl1},\ref{nl2},\ref{nl3}), the discussion is analogous to that in \cite{BGGIY}: equilibria other than the TVF or the Couette flow correspond to bifurcating time-periodic solutions of the Navier-Stokes equations, namely the classical wavy vortices.
The more exotic solutions, for which $r_0(y)$ and $r_1(y)$ are periodic, have the property that the corresponding phases $\theta_0(y)$ and $\theta_1(y)$ are each given by the sum of a linear term and a periodic term with the same period in $y$. As shown in \cite {BGGIY}, these quasi-periodic solutions persist for the full Navier-Stokes system.


\section{Appendix}



\subsection{Computation of coefficient $c$} \label{calcul c}


We now turn our attention to the $y$-independent velocity field $U$. 
After dropping the hats for notational convenience, 
the governing equations \eqref{nl1}, \eqref{nl2}, and \eqref{nl3} lead to
\begin{eqnarray*}
\frac{\partial U}{\partial t} &=&\Delta _{\bot }U-\nabla _{\bot }p+\left( 
\begin{array}{c}
Tg(x)u_{y} \\ 
u_{x} \\ 
0%
\end{array}%
\right) -(U_{\bot }\cdot \nabla _{\bot })U+\frac{T}{2}(1-\mu)\left( 
\begin{array}{c}
u_{y}^2 \\ 
0 \\ 
0%
\end{array}%
\right), \\
\end{eqnarray*}%
$$
\nabla _{\bot }\cdot U_{\bot } =0.
$$
Here $U=(U_{\bot},u_y)$, and the subscript $\bot$ denotes the components in the $(x,z)$-plane.
The field $U$ is $2\pi
/\alpha _{c}$ periodic in $z,$ and satisfies the $0$ boundary conditions at $%
x=\pm 1/2.$ Let us define the operators $\mathbf{L}_0$, $\mathbf{L}_1$, and the quadratic form $\mathbf{B}$ by
\[
\mathbf{L}_{0}U:=\Delta _{\bot }U-\nabla _{\bot }p_{0}+\left( 
\begin{array}{c}
T_{c}g(x)u_{y} \\ 
u_{x} \\ 
0%
\end{array}%
\right) ,
\]%
where $p_{0}$ is such that 
\[
\nabla _{\bot }\cdot (\mathbf{L}_{0}U)_{\bot }=0,
\]
\[
(\mathbf{L}_{0}U)_{x}=0 \quad \hbox{when} \quad x=\pm 1/2;
\]%
\[
\mathbf{L}_{1}U:=-\nabla _{\bot }p_{1}+\left( 
\begin{array}{c}
g(x)u_{y} \\ 
0 \\ 
0%
\end{array}%
\right) ,
\]%
with $p_{1}$ such that%
\[
\nabla _{\bot }\cdot (\mathbf{L}_{1}U)_{\bot }=0,
\]
\[
(\mathbf{L}_{1}U)_{x}=0 \quad \hbox{when} \quad x=\pm 1/2,
\]
and
\begin{eqnarray*}
\mathbf{B}(U,V) &:=&-\frac{1}{2} \Bigl[ (U_{\bot }\cdot \nabla _{\bot })V+(V_{\bot
}\cdot \nabla _{\bot })U \Bigr]+\frac{T}{2}(1-\mu)\left( 
\begin{array}{c}
u_{y}v_{y} \\ 
0 \\ 
0%
\end{array}%
\right)+\nabla _{\bot }q, 
\end{eqnarray*}%
\[
\nabla _{\bot }\cdot (\mathbf{B}(U,V))_{\bot } =0,
\]
\[ 
(\mathbf{B}(U,V))_{x}=0 \quad \hbox{when} \quad x=\pm 1/2.
\]
The system now takes the form
\begin{equation}
\frac{\partial U}{\partial t}=\mathbf{L}_{0}U+\tau \mathbf{L}_{1}U+\mathbf{B}_c%
(U,U)+\tau\mathbf{B}_1(U,U),  \label{NSequ}
\end{equation}
where $\mathbf{B}_c=\mathbf{B}|_{T=T_c}$. %
The equation is posed in a space of divergence-free vector fields satisfying the zero boundary conditions, with
\[
\tau =T-T_{c}.
\]%
To obtain the Landau equation that describes the first bifurcation occurring
for $\tau $ close to $0$, we use the fact that $\mathbf{L}_{0}$ has a double 
zero eigenvalues with corresponding eigenvectors%
\[
\zeta =e^{i\alpha _{c}z}\widehat{U}(x),\text{   }\overline{\zeta }=e^{-i\alpha _{c}z}%
\overline{\widehat{U}}(x),
\]%
and the fact that the rest of the spectrum of  operator $\mathbf{L}_{0}+\tau \mathbf{L}_{1}$ is only composed of 
isolated eigenvalues of negative real parts, not close to 0. Moreover, the center manifold reduction
applies (see \cite{Haragus}) with a symmetry $O(2)$, and the
dynamics near $0$ reduces to the study of a 2-dimensional differential equation in $
\mathbb{C}$:
\begin{equation}
\frac{dA}{dt}=a\tau A-cA|A|^{2},  \label{Landauequ}
\end{equation}
where 
\begin{eqnarray}
U &=&A(t)\zeta +\overline{A}(t)\overline{\zeta }+\Phi (A,\overline{A},\tau ),
\label{centermanif} \\
\Phi (A,\overline{A},\tau ) &=&\tau A\Phi _{10}^{(1)}+A^{2}\Phi
_{20}+|A|^{2}\Phi _{11}+\overline{A}^{2}\overline{\Phi _{20}}+...  \nonumber
\end{eqnarray}%
The coefficients $\Phi _{ij}$ are divergence free vector functions of $x$
which satisfy the boundary conditions. The coefficients $a$, $c$ and $\Phi
_{ij}$ may be computed as indicated below. Replacing $U$ by (\ref%
{centermanif}) in (\ref{NSequ}) and using (\ref{Landauequ}), gives after
identification of each monomial $\tau ^{p}A^{n}\overline{A}^{m}$%
\begin{eqnarray}
a\zeta  &=&\mathbf{L}_{0}\Phi _{10}^{(1)}+\mathbf{L}_{1}\zeta 
\label{coef a} \\
0 &=&\mathbf{L}_{0}\Phi _{20}+\mathbf{B}_c(\zeta ,\zeta ),  \nonumber \\
0 &=&\mathbf{L}_{0}\Phi _{11}+2\mathbf{B}_c(\zeta ,\overline{\zeta }), 
\nonumber \\
-c\zeta  &=&\mathbf{L}_{0}\Phi _{21}+2\mathbf{B}_c(\zeta ,\Phi _{11})+2\mathbf{B%
}_c(\overline{\zeta },\Phi _{20}).  \label{coef c}
\end{eqnarray}%
The scalar product is defined as
\[
\langle U,V\rangle =\int_{0}^{2\pi /\alpha_c }\int_{-1/2}^{1/2}(u_{x}\overline{%
v_{x}}+u_{y}\overline{v_{y}}+u_{z}\overline{v_{z}})\, dxdz,
\]%
and we define the adjoint operator $\mathbf{L}_{0}^{\ast }$ by%
\[
\mathbf{L}_{0}^{\ast }V=\Delta _{\bot }V-\nabla _{\bot }q_{0}+\left( 
\begin{array}{c}
v_{y} \\ 
T_{c}g(x)v_{x} \\ 
0%
\end{array}%
\right) .
\]%
Since we have%
\[
\mathbf{L}_{0}\zeta =0,
\]%
we obtain%
\[
a\langle \zeta ,\zeta ^{\ast }\rangle =\langle \mathbf{L}_{1}\zeta ,\zeta
^{\ast }\rangle ,
\]%
\[
-c\langle \zeta ,\zeta ^{\ast }\rangle =\langle 2\mathbf{B}_c(\zeta ,\Phi
_{11})+2\mathbf{B}_c(\overline{\zeta },\Phi _{20}),\zeta ^{\ast }\rangle ,
\]%
where $\zeta ^{\ast }$ is the eigenvector of the form $e^{i\alpha _{c}z}V(x)$
for the $0$ eigenvalue of the adjoint operator $\mathbf{L}_{0}^{\ast }.$ The
eigenvectors $\zeta $ and $\zeta ^{\ast }$ satisfy 
\begin{eqnarray*}
\zeta  &=&e^{i\alpha _{c}z}(u_{x}^{0},u_{y}^{0},u_{z}^{0})^{t}, \\
\zeta ^{\ast } &=&e^{i\alpha _{c}z}(v_{x}^{0},v_{y}^{0},v_{z}^{0})^{t},
\end{eqnarray*}%
with%
\begin{eqnarray}
(D^{2}-\alpha _{c}^{2})u_{x}^{0}-Dp_{0}+T_{c} g(x) u_{y}^{0} &=&0, \label{sys1} \\
(D^{2}-\alpha _{c}^{2})u_{y}^{0}+u_{x}^{0} &=&0, \label{sys2} \\
(D^{2}-\alpha _{c}^{2})u_{z}^{0}-i\alpha _{c}p_{0} &=&0, \label{sys3} \\
Du_{x}^{0}+i\alpha _{c}u_{z}^{0} &=&0, \label{sys4}
\end{eqnarray}%
and the adjoint system
\begin{eqnarray}
(D^{2}-\alpha _{c}^{2})v_{x}^{0}-Dq_{0}+v_{y}^{0} &=&0, \label{sys11} \\
(D^{2}-\alpha _{c}^{2})v_{y}^{0}+ T_{c} g(x) v_{x}^{0} &=&0, \label{sys12} \\
(D^{2}-\alpha _{c}^{2})v_{z}^{0}-i\alpha _{c}q_{0} &=&0, \label{sys13} \\
Dv_{x}^{0}+i\alpha _{c}v_{z}^{0} &=&0, \label{sys14}
\end{eqnarray}%
with the usual boundary conditions at $x=\pm 1/2.$ We observe that $%
u_{x}^{0},u_{y}^{0},v_{x}^{0},v_{y}^{0}$ are real functions, while $%
u_{z}^{0},v_{z}^{0}$ are pure imaginary.

For the calculation of $c$, we need to compute $\Phi _{11}$ and $\Phi _{20}.$
Let us denote%
\[
\Phi _{11}=(u_{x}^{11},u_{y}^{11},u_{z}^{11})^{t},
\]%
and%
\[
(\zeta _{\bot }\cdot \nabla _{\bot })\overline{\zeta }+(\overline{\zeta }%
_{\bot }\cdot \nabla _{\bot })\zeta =\left( 
\begin{array}{c}
2D(u_{x}^{0})^{2} \\ 
2D(u_{x}^{0}u_{y}^{0}) \\ 
0%
\end{array}%
\right) ,
\]%
then we need to solve%
\begin{eqnarray*}
D^{2}u_{x}^{11}-Dp + T_{c} g(x) u_{y}^{11} &=&2D(u_{x}^{0})^{2}-(1-\mu)T_c(u_y^0)^2 \\
D^{2}u_{y}^{11}+u_{x}^{11} &=&2D(u_{x}^{0}u_{y}^{0}) \\
D^{2}u_{z}^{11} &=&0, \\
Du_{x}^{11} &=&0
\end{eqnarray*}%
with the usual boundary conditions. This implies%
\[
\Phi _{11}=(0,u_{y}^{11},0)^{t},
\]%
with%
\[
D^{2}u_{y}^{11}=2D(u_{x}^{0}u_{y}^{0}), \qquad  u_{y}^{11}(\pm 1/2)=0.
\]%
Hence $u_{y}^{11}$ is real and is given by%
\begin{equation}
u_{y}^{11}(x)=2\int_{-1/2}^{x}u_{x}^{0}(s)u_{y}^{0}(s)ds-2(x+1/2)%
\int_{-1/2}^{1/2}u_{x}^{0}(s)u_{y}^{0}(s)ds.  \label{Phi11}
\end{equation}%
Let us now denote%
\[
\Phi _{20}=e^{2i\alpha_c z}(u_{x}^{20},u_{y}^{20},u_{z}^{20})^{t}
\]%
and compute%
\[
(\zeta _{\bot }\cdot \nabla _{\bot })\zeta =e^{2i\alpha_c z}\left( 
\begin{array}{c}
0 \\ 
u_{x}^{0}Du_{y}^{0}-u_{y}^{0}Du_{x}^{0} \\ 
\frac{i}{\alpha_c }u_{x}^{0}D^{2}u_{x}^{0}-\frac{i}{\alpha_c }(Du_{x}^{0})^{2}%
\end{array}%
\right).
\]%
We need to solve%
\begin{eqnarray*}
(D^{2}-4\alpha_c ^{2})u_{x}^{20}-Dp + T_{c} g(x) u_{y}^{20} &=&-\frac{T_c}{2}(1-\mu)(u_y^0)^2, \\
(D^{2}-4\alpha_c ^{2})u_{y}^{20}+u_{x}^{20}
&=&u_{x}^{0}Du_{y}^{0}-u_{y}^{0}Du_{x}^{0}, \\
(D^{2}-4\alpha_c ^{2})u_{z}^{20}-2i\alpha_c p &=&\frac{i}{\alpha_c }%
[u_{x}^{0}D^{2}u_{x}^{0}-(Du_{x}^{0})^{2}], \\
Du_{x}^{20}+2i\alpha_c u_{z}^{20} &=&0,
\end{eqnarray*}%
\[
u_{x}^{20}=Du_{x}^{20}=u_{y}^{20}=0,x=\pm 1/2.
\]%
This leads to a sixth-order system%
\begin{eqnarray*}
(D^{2}-4\alpha_c ^{2})^{2}u_{x}^{20}-4\alpha_c ^{2}T_{c} g(x) u_{y}^{20}
&=&2D[u_{x}^{0}D^{2}u_{x}^{0}-(Du_{x}^{0})^{2}]+2\alpha_c^2T_c(1-\mu)(u_y^0)^2, \\
(D^{2}-4\alpha_c ^{2})u_{y}^{20}+u_{x}^{20}
&=&u_{x}^{0}Du_{y}^{0}-u_{y}^{0}Du_{x}^{0}, \\
u_{x}^{20} &=&Du_{x}^{20}=u_{y}^{20}=0,x=\pm 1/2.
\end{eqnarray*}%
This yields that $u_{x}^{20}$ and $u_{y}^{20}$ are real, while $u_{z}^{20}$ is purely imaginary. Now we can compute%
\[
2\mathbf{B}_c(\zeta ,\Phi _{11})=
\Pi e^{i\alpha_c z}\left( 
\begin{array}{c}
T_c (1-\mu) u_{y}^0 u_{y}^{11} \\ 
-u_{x}^{0}Du_{y}^{11} \\ 
0%
\end{array}%
\right) 
.
\]
We also obtain%
\begin{align*}
2\mathbf{B}_c(\overline{\zeta },\Phi _{20})&=-\Pi e^{i\alpha_c z}\left( 
\begin{array}{c}
D(u_{x}^{0}u_{x}^{20})+2u_{x}^{20}Du_{x}^{0}+\frac{1}{2}u_{x}^{0}Du_{x}^{20}-T_c (1-\mu) u_{y}^0 u_{y}^{20}
\\ 
u_{x}^{0}Du_{y}^{20}+2u_{y}^{20}Du_{x}^{0}+u_{x}^{20}Du_{y}^{0}+\frac{1}{2}%
u_{y}^{0}Du_{x}^{20} \\ 
\frac{i}{2\alpha_c }[D(u_{x}^{0}Du_{x}^{20})-2u_{x}^{20}D^{2}u_{x}^{0}]%
\end{array}%
\right).
\end{align*}
Hence, from (\ref{coef c}), as the projection $\Pi$ disappears in the scalar product,
we obtain $c$ through
\begin{eqnarray*}
&&c\int_{-1/2}^{1/2}\Bigl[ u_{x}^{0}v_{x}^{0}+u_{y}^{0}v_{y}^{0}+\frac{1}{\alpha
_{c}^{2}} Du_{x}^{0} Dv_{x}^{0} \Bigr] \, dx 
\\
&=&%
\int_{-1/2}^{1/2}v_{y}^{0} \Bigl[ u_{x}^{0}Du_{y}^{11}+u_{x}^{0}Du_{y}^{20}+2u_{y}^{20}Du_{x}^{0}+u_{x}^{20}Du_{y}^{0}+%
\frac{1}{2}u_{y}^{0}Du_{x}^{20} \Bigr] \, dx 
\\
&&+\int_{-1/2}^{1/2}v_{x}^{0} \Bigl[ D(u_{x}^{0}u_{x}^{20})+2u_{x}^{20}Du_{x}^{0}+%
\frac{1}{2}u_{x}^{0}Du_{x}^{20} - T_c(1-\mu)u_y^0 (u_y^{11}+u_y^{20}) \Bigr] dx 
\\
&&+\int_{-1/2}^{1/2}\frac{1}{2\alpha_{c}^{2}}%
Dv_{x}^{0} \Bigl[ D(u_{x}^{0}Du_{x}^{20})-2u_{x}^{20}D^{2}u_{x}^{0} \Bigr] \, dx.
\end{eqnarray*}%


\subsection{Numerical computation of $c$} \label{App:calcul c}


Let us now detail the adjoint system. We apply $\alpha_c^2 - D^2$ to (\ref{sys14}) and use (\ref{sys13}), which leads to
$$
(\alpha_c^2 - D^2) D v_x^0 + \alpha_c^2 q_0 = 0.
$$
The adjoint system when $\lambda = 0$ is thus
\begin{eqnarray*}
(D^2 - \alpha_c^{2})^2 v_{x} &=& \alpha_c ^{2} v_y \\
(D^{2} - \alpha_c^2)v_y &=& -  T_c g(x) v_{x},
\end{eqnarray*}
namely
\begin{equation} \label{ad1}
D^4 v_x = - \alpha_c^4 v_x + 2 \alpha_c^2 D^2 v_x + \alpha_c^2  v_y,
\end{equation}
\begin{equation} \label{ad2}
D^2 v_y = \alpha_c^2 v_y - T_c g(x) v_x .
\end{equation}
For $\mu> \mu_c \approx -0.8$, the coefficient $\sigma=b_4$ is still positive, 
while the coefficient $c$ has changed its sign, passing from positive to negative at $\hat \mu_c \approx-0.65$; see Figure \ref{signofc}. 
Consequently, the study made in section \ref{G-L equ1} is relevant for $\mu_c<\mu< \hat \mu_c$. For $\mu$ close to $\mu_c$, the relevant study is in section \ref{G-L equ2}. 

We note that this configuration of $c$ and $b_4$ is consistent with the computations for the full Couette-Taylor system with $\eta=0.95$ in the book \cite{Iooss}: see p.41 Fig III.2  for the sign of $c$, and p.61 Fig IV.2 for the sign of $b_4$. In the latter case, the absence of a sign change corresponds to the first unstable mode $m=0$, and the parameter $\mathfrak{R}_0=\frac{\omega_0r_0d}{\nu}$ used in \cite{Iooss} is related to our parameter $T$ by $T=\frac{2(1-\mu)(1-\eta)}{\mu^2}\mathfrak{R}_{0}^2$.
\begin{figure}[th]
\begin{center}
\includegraphics[width=10cm]{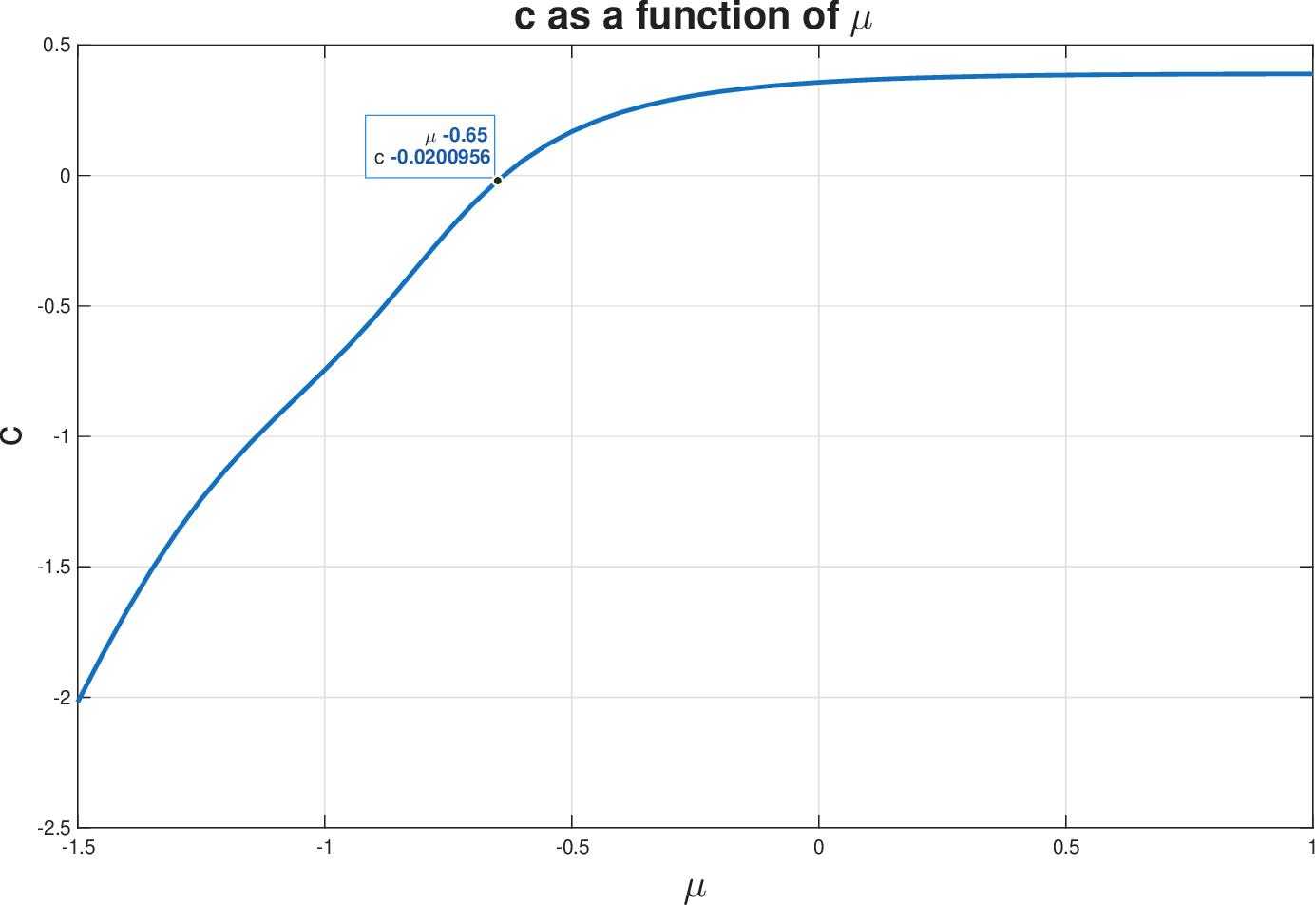}
\end{center}
\caption{Sign of $c$}
\label{signofc}
\end{figure}


\subsubsection*{Acknowledgements}


D. Bian is supported by NSFC under the contract 12271032. Z. Yang is partially supported by the NSF grant DMS-2550221.


\end{document}